\documentclass[reqno,11pt]{amsart}
\usepackage[colorlinks=true, linkcolor=blue, citecolor=blue]{hyperref}
    
\usepackage{amssymb}
\usepackage{amsmath, graphicx, rotating}
\usepackage{color}
\usepackage{soul}
\usepackage[dvipsnames]{xcolor}
           
\usepackage{ifthen}
\usepackage{xkeyval}
\usepackage{todonotes}
\usepackage[T1]{fontenc}
\usepackage{lmodern}
\usepackage[english]{babel}

\usepackage{ upgreek }
\usepackage{stmaryrd}
\SetSymbolFont{stmry}{bold}{U}{stmry}{m}{n}
\usepackage{amsthm}
\usepackage{float}

\usepackage{ bbm }
\usepackage{ stmaryrd }
\usepackage{ mathrsfs }
\usepackage{ frcursive }
\usepackage{ comment }

\usepackage{pgf, tikz}
\usetikzlibrary{shapes}
\usepackage{varioref}
\usepackage{enumitem}

\usepackage{mathtools}

\definecolor{rouge}{rgb}{0.7,0.00,0.00}
\definecolor{vert}{rgb}{0.00,0.5,0.00}
\definecolor{bleu}{rgb}{0.00,0.00,0.8}

\usepackage[margin=1.5in, tmargin=1.3in, bmargin=1.3in]{geometry}  

\newtheorem{theorem}{Theorem}[section]
\newtheorem*{theorem*}{Theorem}
\newtheorem{lemma}[theorem]{Lemma}

\newtheorem{corollary}[theorem]{Corollary}
\newtheorem{proposition}[theorem]{Proposition}

\labelformat{hypothesis}{\textbf{M\kern-0.1mm#1}}

\newtheorem{condition}{Condition}
\newtheorem{conditionA}{A\kern-0.1mm}
\labelformat{conditionA}{\textbf{A\kern-0.1mm#1}}

\renewcommand\dots{\hbox to 1em{.\hss.\hss.}}

\theoremstyle{definition}
\newtheorem{example}[theorem]{Example}
\newtheorem{remark}[theorem]{Remark}

\numberwithin{equation}{section}

\def\bb#1{\mathbb{#1}}
\def\bf#1{\mathbf{#1}}
\def\bbm#1{\mathbbm{#1}}
\def\geq{\geqslant}
\def\leq{\leqslant}

\newcommand\ee{\varepsilon}
\DeclareMathOperator{\supp}{supp}

\DeclarePairedDelimiter\floor{\lfloor}{\rfloor}

\begin{document}

\title[A zero-one law for invariant measures]
{A zero-one law for invariant measures \\ and a local limit theorem for coefficients of \\ 
random walks on the general linear group}


\author{Ion Grama}
\author{Jean-Fran\c cois Quint}
\author{Hui Xiao}

\curraddr[Grama, I.]{ Universit\'{e} de Bretagne-Sud, CNRS 6205, Vannes, France.}
\email{ion.grama@univ-ubs.fr}

\curraddr[Quint, J.-F.]{ CNRS-Universit\'{e} de Bordeaux I, 33405, Talence, France.}
\email{Jean-Francois.Quint@math.u-bordeaux.fr}

\curraddr[Xiao, H.]{ Universit\'{e} de Bretagne-Sud, CNRS 6205, Vannes, France.}
\email{hui.xiao@univ-ubs.fr}


\begin{abstract}
We prove a zero-one law for the stationary measure for algebraic sets
generalizing the results of Furstenberg \cite{Furst BTSPHS-73} 
and Guivarc'h and Le Page \cite{GL16}. 
As an application, we establish a local limit theorem for 
the coefficients of random walks on the general linear group. 
\end{abstract}

\date{\today}
\maketitle



\section{Introduction}\label{sec-intro}

\subsection{Motivation and objectives}
Let $d\geq 2$ be an integer. 
We denote by $\mathbb{R}^d$ the $d$-dimensional Euclidean space equipped with the standard Euclidean scalar product 
and by $(\mathbb R^d)^*$ the dual space of $\mathbb{R}^d$. 
Let  $\ell$ be the Lebesgue measure on $\mathbb R.$ 
For any $v\in \mathbb R^d $ and $ f \in (\mathbb R^d)^*$ the corresponding 
duality bracket is denoted by 
$\langle f, v \rangle = f(v)$; 
sometimes instead of $f(v)$ we shall use the reverse notation $ v(f) = \langle v, f \rangle$.
Set $\mathbb R^d_0=\mathbb R^d \setminus \{ 0 \}$. 
Denote by $\mathbb G=GL(d,\mathbb R)$ the general linear group of 
invertible $d\times d$ matrices with coefficients in $\mathbb R$.
 The projective space $\mathbb P^{d-1}$ of $\mathbb R^d$ 
is the set of elements $x=\mathbb R v$, 
where $v\in \mathbb R^d_0$.
The projective space of $(\mathbb R^d)^*$ is denoted by $(\mathbb P^{d-1})^*$. 
For any $x=\mathbb Rv \in \mathbb P^{d-1}$ and $y=\mathbb Rf \in (\mathbb P^{d-1})^*$ we define
$\delta(y,x)=  \frac{| \langle f,v \rangle |}{|f| |v|}$.
For any $g\in \mathbb G$ and $x = \mathbb R v \in\mathbb P^{d-1}$ with $v\in\mathbb R^d_0$
let $gx 
= \mathbb R  gv \in \mathbb P^{d-1}$, 
while $gv\in \mathbb R^d$ is the image of the automorphism 
$v\mapsto gv$ on $\mathbb R^d$. 
Set $\mathbb N = \{ 0,1,2,\ldots  \} $ and  $\mathbb N^*=\mathbb N \setminus \{ 0\}$.

Let $\mu$ be a probability measure on $\mathbb G$.  
Consider the probability space $(\Omega,\mathscr F, \mathbb P),$ where $\Omega= \mathbb G^{\mathbb N^*}$,
$\mathscr F$ is the Borel $\sigma$-algebra on $\Omega$ 
and $\mathbb P= \mu^{\otimes \mathbb N^*}$.
If we denote by $g_i$ the coordinate mapping on $\Omega$, 
then $g_1,g_2,\ldots$ is a sequence of i.i.d.\ random elements in $\mathbb G$ 
defined on $(\Omega,\mathscr F, \mathbb P)$ with the same law $\mu$.
Set $G_n=g_n \ldots g_1$, for $n\geq 1$.

A measure $\nu$ is called $\mu$-invariant if $\mu*\nu=\nu$, where $*$ stands for the convolution of probability measures. 
Furstenberg \cite{Furst BTSPHS-73} 
showed that under some mild conditions there exists a $\mu$-invariant probability measure $\nu$
on $\mathbb P^{d-1}$ which is not supported by any proper projective hyperplane: 
for any projective hyperplane $Y\subsetneq \mathbb P^{d-1}$, 
\begin{align} \label{eqfurstenb001}
\nu(Y) =0.
\end{align}
The Furstenberg zero-law \eqref{eqfurstenb001} turns out to be one of the key properties 
in the study of products of random matrices. 
It is used  
in the proof of many limit theorems 
for the norm cocycle  $\log |G_n v|$,
where $v\in \mathbb R^d_0$ is a starting vector.
We refer to Furstenberg and Kesten \cite{FK60}, Le Page \cite{LeP82}, Bougerol and Lacroix \cite{BL85}, 
Guivarc'h \cite{Gui15}, Benoist and Quint \cite{BQ14-RWPS, BQ16b},
who established the law of large numbers,  the central limit theorem, the local limit theorem 
and large deviation asymptotics. 
Some of these results have been extended to the coefficients 
$\langle f,G_n v \rangle$, where $f \in (\mathbb R^d)^*$, however much less is known in this respect.
Guivarc'h and Raugi \cite{GR85} have proved the law of large numbers and the central limit theorem for the coefficients. 
In the setting of reductive groups, Benoist and Quint \cite{BQ16b} 
have established the law of iterated logarithm and the  large deviation bounds.
The approach developed in \cite{GR85}, \cite{BL85} and \cite{Gui90} for the proof of 
the law of large numbers and of the central limit theorem for $\log |\langle f,G_n v \rangle |$ 
is based on the use of the quantitative version of the property \eqref{eqfurstenb001}
called H\"older regularity of the stationary measure $\nu$ 
which we state below: there exist two positive constants $\alpha$ and $C$ such that, 
for any $y = \mathbb R f \in (\mathbb P^{d-1})^*$ and $\varepsilon > 0$,
\begin{align} \label{eq-reg001}
\nu \left( \left\{x \in \mathbb P^{d-1}: \delta(y,x) \leq \varepsilon \right\} \right) \leq C\varepsilon^{\alpha}.
\end{align}
The next elementary identity 
relates the coefficient 
$ \langle f,G_n v \rangle $ 
to the norm $|G_n v|$:  
for any $x=\mathbb Rv \in \mathbb P^{d-1}$ and $y=\mathbb Rf \in (\mathbb P^{d-1})^*$ with 
$|f|=1$,
\begin{align} \label{eq-represent001}
\log |\langle f,G_n v \rangle | = \log |G_nv| + \log\delta(y, G_n x),
\end{align}
where $\delta(y, G_n x)  =  \frac{| \langle f, G_n v   \rangle |}{|f| |G_n v|}$.
From \eqref{eq-reg001} one can deduce that for any $\beta > 0$,
\begin{align} \label{eq-almost sure 001}
\lim_{n \to \infty} n^{-\beta} \log\delta(y, G_n x) = 0, \quad \mathbb P\mbox{-a.s..}
\end{align}
Now using \eqref{eq-represent001} and \eqref{eq-almost sure 001} we can infer the limit behaviour of  
$\log |\langle f,G_n v \rangle |$ from that of $\log |G_n v|$. 
This allows for instance to prove the law of large numbers  and the central limit theorem: 
for any $x=\mathbb Rv \in \mathbb P^{d-1}$ and $y=\mathbb Rf \in (\mathbb P^{d-1})^*$,
\begin{align*} 
\lim_{n \to \infty} \frac{1}{n}\log |\langle f,G_n v \rangle | = \lambda, \quad \mathbb P\mbox{-a.s.}
\end{align*}
and
\begin{align*} 
\lim_{n\to \infty} \mathbb P \left( \frac{\log |\langle f,G_n v \rangle | - n\lambda}{\sigma\sqrt{n}} \leq t \right) = \Phi(t),
\end{align*}
where $\lambda \in \bb R$ is a constant called the Lyapunov exponent, 
$\sigma^2=\lim_{n\to \infty}\frac{1}{n}\mathbb E (\log |G_n v|-n\lambda)^2>0$ is a positive number 
and $\Phi$ is the standard normal distribution function.
For the law of iterated logarithm and large deviation bounds, 
Benoist and Quint \cite[Lemma 14.11]{BQ16b}, 
following the approach of Bourgain, Furman, Lindenstrauss and Mozes \cite{BFLM11}, 
have developed another strategy 
based on the following inequality:
for any $a >0$ there exist positive constants $c$ and $n_0$ such that for any $n_0\leq l \leq n$, 
$y \in (\mathbb P^{d-1})^*$ and $x \in \mathbb P^{d-1}$, 
\begin{align}\label{Regularity_Ineq 001}
\mathbb P \left( \delta( y, G_n x) \leq  e^{-a l}  \right) \leq C e^{- cl}.
\end{align}
The bound \eqref{Regularity_Ineq 001} implies \eqref{eq-reg001}, 
and therefore contains more information than  \eqref{eq-reg001}.
It gives an alternative way to prove the law of large numbers and the central limit theorem, 
but also allows
to establish new results like the law of iterated logarithm and the large deviation bounds.    
However, many important properties such as the Berry-Esseen bounds, local limit theorems, 
large deviation principles and exact asymptotics of large deviations for the coefficients $\langle f,G_n v \rangle$ 
cannot be obtained by this approach.  
For these latter statements the exact contribution of the term $\delta(y, G_n x)$ should be accounted,
which means that we need to establish more general theorems for the couple
$G_n x$ and  $\log |G_n v |$.

We find out that in order to transfer many asymptotic properties 
from the couple $(G_n x, \log |G_n v |)$ to that of the 
coefficients $\log |\langle f,G_n v \rangle |$ 
it is necessary to establish the identity \eqref{eqfurstenb001} for subsets 
of $\mathbb P^{d-1}$ which are not projective hyperplanes, 
in particular, for the hypersurfaces $\{x\in \mathbb P^{d-1}: \delta(y,x)=t\},$ where $t\not = 0$.
The main goal of the paper is to extend the result of Furstenberg \eqref{eqfurstenb001}
from the special case of  projective hyperplanes 
to arbitrary algebraic subsets $Y$ of the projective space $\mathbb P^{d-1}$. 
There is, however, an essential difference with the Furstenberg's result, 
 which confers the mass $0$ to a projective hyperplane. 
We show that for an arbitrary algebraic set $Y$ of $\mathbb P^{d-1}$ 
it holds that $\nu(Y)$ is $0$ or $1$.  
Contrary to the Furstenberg zero-law, 
it is possible, as we show in Example \ref{Example1}, that the invariant measure $\nu$ is concentrated on 
an algebraic subset of dimension $d-2$ on the projective space $\mathbb P^{d-1}$.  
It is also interesting to note that for projective hyperplanes the Furstenberg zero-law can be strengthened to
 the regularity property \eqref{eq-reg001}, 
while for algebraic sets the quantitative analog of \eqref{eq-reg001} has not yet been established. 
%
 
Using the zero-one law of the stationary measure $\nu$ we will prove the following local limit theorem:
under appropriate conditions, for any real numbers $a_1 < a_2$, $f\in (\mathbb R^d)^*$ and $v\in \mathbb R^d$, 
as $n \to \infty$,
\begin{align*}
\mathbb{P} \Big(  \log |\langle f, G_nv \rangle|  - n \lambda \in [a_1, a_2]  \Big) 
= \frac{a_2 - a_1}{ \sigma \sqrt{2 \pi n} } (1+o(1)). 
\end{align*}

Based on a zero-one law for the invariant measure 
under the change of measure 
which is proved in Theorem \ref{Th-main-subsets001}, 
it is possible to establish a local limit theorem with large deviations for the coefficients of 
$G_n$, however this will be done elsewhere.

\subsection{Idea of the proof of the local limit thorem} 
For illustration we show how to apply our zero-one  law to establish the local limit theorem for the coefficients 
$\langle f, G_n v \rangle$, where $f\in (\mathbb R^d)^*$ and $v\in \mathbb R^d$.  
Letting $I=[a_1,a_2]$ be an interval of the real line,  
we have to handle the probability
\begin{align} \label{eq Largedevprob001}
\mathbb P(\log |\langle f, G_n v \rangle| - n \lambda \in I). 
 \end{align}
Using \eqref{eq-represent001} and discretizing the values of $\delta(y, G_n x)$, the probability 
\eqref{eq Largedevprob001} is bounded from above by
\begin{align} \label{eq Largedevprob002}
\sum_{k}\mathbb P \Big( \log |G_n v | - n \lambda  \in I_{\eta} + \eta k ,  G_n x \in Y_k^{\eta}  \Big), 
\end{align}
where $\eta>0$ is sufficiently small and will be chosen later, $I_{\eta} = [a_1-\eta, a_2+\eta]$
and $Y_k^{\eta} = \{ x \in \mathbb P^{d-1} : - \log \delta(y,x) \in \eta [k-1, k) \}$.
By the local limit theorem for the couple $G_n x$ and  $\log |G_n v |$
(actually in the paper we circumvent it by using the spectral 
gap theory 
and some smoothing technique) 
each probability in the sum \eqref{eq Largedevprob002}  is asymptotically bounded by 
$\frac{\ell(I_{\eta})}{ \sigma \sqrt{2\pi n} } \nu( \overline Y_k^{\eta} )$,
with $\overline Y_k^{\eta} = \{ x \in \mathbb P^{d-1} : - \log \delta(y,x) \in \eta [k-1-\varepsilon, k+\varepsilon) \}$. 
Hence, the sum \eqref{eq Largedevprob002} asymptotically does not exceed 
\begin{align} \label{eq Largedevprob003}
\frac{\ell (I_{\eta})}{ \sigma \sqrt{2\pi n} } \sum_{k} \nu(\overline Y_k^{\eta}). 
\end{align}
An important issue is to show that the sum in \eqref{eq Largedevprob003}  
converges to $1$ as  $\varepsilon\to 0$ and $\eta\to 0$. 
This turns out to be a difficult problem. 
By some easy calculations it reduces to showing that 
for any $k\geq 0$, 
\begin{align} \label{ergular-k001}
\nu \left( \left\{ x \in \mathbb P^{d-1} : \log\delta(y,x) = -\eta k  \right\} \right) =0.
\end{align}
The set $Y_0=\{ x \in \mathbb P^{d-1} : \delta(y,x) = 0  \}$  
is a projective hyperplane in $\mathbb P^{d-1}$ of dimension $d-2$.  
By \eqref{eqfurstenb001},  $\nu(Y_0)=0$,
under the strong irreducibility condition on the measure $\mu$.
For $k \geq 0$, the equality \eqref{ergular-k001} 
may not be true for an arbitrary $\eta$, as we show in the paper. 
In fact we establish a zero-one law for the invariant measure $\nu$ (see Theorem \ref{Lemma-Fursten-set001}), 
from which it follows that
for any $t<0$,
\begin{align*}
\nu \left( \left\{ x \in \mathbb P^{d-1} : \log\delta(y,x) = t  \right\} \right) = 0 \ \mbox{or} \ 1.  
\end{align*}
This statement implies that \eqref{ergular-k001} holds true for any $k \geq 1$ 
if we choose an appropriate constant $\eta$. 
Indeed, if there exists $t<0$ such that
$\nu(\{ x \in \mathbb P^{d-1} : \log\delta(y,x) = t  \} )=1$, then we can choose
$\eta$ such that $-\eta k \not = t$ for any $k$, so that \eqref{ergular-k001} holds true for any $k$.
Otherwise $\eta$ can be chosen arbitrarily. 
This proves that the sum in \eqref{eq Largedevprob003} converges to $1$
as $\varepsilon\to 0$ and $\eta\to 0$ and so we obtain that $\limsup_{n\to\infty}$ 
of the probability \eqref{eq Largedevprob001} 
is bounded from above by $\frac{\ell (I_{\eta})}{\sigma \sqrt{2\pi n}} $.
By some similar reasoning $\liminf_{n\to\infty}$ of \eqref{eq Largedevprob001} is bounded from below by the same quantity. 


\section{Main results}\label{sec-main res}
 
The inverse of $g\in \mathbb G$ is denoted by $g^{-1}$ and the identity element of $\mathbb G$ is the unit matrix $\bf e$.
The adjoint operator $g^*$ of $g\in \mathbb G$
is the automorphism $g^*$ of  $(\mathbb R^d)^*$
defined by  $(g^*f) (v) = f(g v)$, 
where $v\in \mathbb R^d$ and $f\in (\mathbb R^d)^*.$
%
%
Let $\mathcal C(\mathbb P^{d-1})$ be the space of complex valued continuous functions on $\mathbb P^{d-1}$. 

All over the paper 
$\mu$ is a probability law on $\mathbb G$.
 Denote by $\supp \mu$ the support of $\mu$ and by  $\Gamma_{\mu}:=[\supp\mu]$ 
the smallest closed subsemigroup of $\mathbb G$ generated by $\supp \mu$.  

A matrix $g\in \mathbb G$ is said to be \emph{proximal} if it has an algebraic simple dominant eigenvalue, 
namely, $g$ has an eigenvalue $\lambda_{g}$ satisfying $|\lambda_{g}| > |\lambda_{g}'|$
for all other eigenvalues $\lambda_{g}'$ of $g$.
It is easy to check that $\lambda_{g} \in \mathbb{R}$. 
We choose $v_g^+$ an eigenvector vector with unit norm $|v_g| = 1$, corresponding to the eigenvalue 
$\lambda_{g}$, which will be called dominant eigenvector of $g$. 
The unique element $x_g^+=\mathbb R v_g^+ \in \mathbb P^{d-1}$ 
will be called 
\emph{attractor} of $g$.
Note that it does not depend on the choice of $v_g^+$.

We need the following strong irreducibility and proximality conditions: 

 \begin{conditionA}[Strong irreducibility] \label{Cond strong-irred}
No finite union of proper subspaces of $\mathbb{R}^d$ is $\Gamma_{\mu}$-invariant.
\end{conditionA}

 \begin{conditionA}[Proximality]  \label{Cond proxim}
$\Gamma_{\mu}$ contains at least one proximal matrix (i.e.  a matrix with a simple dominant eigenvalue).
\end{conditionA}

%

By definition the strong irreducibility here concerns 
the action of the matrices on the Euclidean space $\mathbb R^d$.
More precisely, it means that there is no finite union
of subspaces $V_1, \ldots, V_k$ in $\mathbb R^d$ such that
 $g(V_1\cup \ldots \cup V_k) \subset V_1\cup \ldots \cup V_k$ for any $g\in \Gamma_\mu$. 

Assume conditions  \ref{Cond strong-irred} and \ref{Cond proxim}. 
By a well known result of Furstenberg \cite{Furst BTSPHS-73}, 
on the projective space $\mathbb P^{d-1}$ there exists a unique $\mu$-stationnary probability measure $\nu$
such that for any $\varphi \in \mathcal C(\mathbb P^{d-1})$,
\begin{align} \label{mu station meas}
 \int_{\mathbb P^{d-1}} \int_{\mathbb G} \varphi(g x) \mu(dg) \nu(dx)
 = \int_{\mathbb P^{d-1}} \varphi(x) \nu(dx). 
\end{align}
Moreover, Furstenberg \cite{Furst BTSPHS-73} 
(see also Bougerol and Lacroix \cite[Chapter III, Proposition 2.3]{BL85})
showed that under appropriate assumptions 
any proper projective subspace $Y \subset \mathbb P^{d-1}$ has $\nu$-measure $0$.

\begin{theorem}[Furstenberg] \label{Lemma-Fursten-set}
Assume condition \ref{Cond strong-irred}. 
Then, for any $\mu$-stationary measure $\nu$ and any proper projective subspace $Y \subset \mathbb P^{d-1}$
it holds that $\nu( Y)=0$.
\end{theorem}


Our first result extends Theorem \ref{Lemma-Fursten-set} to algebraic subsets of $\mathbb P^{d-1}$.
We recall that a subset $X$ in $\mathbb R^d$ is algebraic if there exist polynomial functions
$p_1,\ldots,p_k$ on $\mathbb R^d$ such that $X= \{v\in \mathbb R^d :  p_1(v)=\ldots =p_k(v)=0 \}$.  
We say that $X$ is homogeneous if for every $t\in \mathbb R $ and $v\in X$ it holds that $tv \in X$. 
A subset $Y$ of $\mathbb P^{d-1}$ is algebraic if there exists an algebraic homogeneous subset
$X$ in $\mathbb R^d$ whose projective image on $\mathbb P^{d-1}$ is $Y$.


\begin{theorem} \label{Lemma-Fursten-set001} 
Assume conditions \ref{Cond strong-irred} and \ref{Cond proxim}. 
Then, for any algebraic subset $Y$ of 
$\mathbb P^{d-1}$ 
it holds that either $\nu( Y)=0$ or $\nu( Y)=1$. 
\end{theorem}

The statement of Theorem \ref{Lemma-Fursten-set001}  can be interpreted as a statement about the behaviour of the 
trajectory $G_n x $: either for some $x\in \mathbb P^{d-1}$ the trajectory $G_n x$ stays in $Y$ with probability $1$, 
or for every $x\in \mathbb P^{d-1}$ this trajectory mostly avoids $Y$.   
Another point of view of the way the random walk $G_n x$ avoids an algebraic subset 
is developed in the paper by Aoun \cite{Aou13}.

The following example shows that there exist proper algebraic subsets of $\mathbb P^{d-1}$
of $\nu$-invariant measure   $0$ or $1$. 
 \begin{example} \label{Example1}
Let $d\geq 3$.  
Fix an integer $p$ such that $1\leq p \leq d-1$ and equip $\mathbb R^d$
with the quadratic form 
$q(v)= v_1^2 + v_2^2 +\ldots + v_p^2 - v_{p+1}^2 - \ldots - v_d^2.$
Let $O(q)$ be the group of isometries of $q$,
that is the group of elements $g\in GL(d,\mathbb R)$ such that 
$q(gv)= q(v)$ for all $v\in \mathbb R^d$.
We choose $\mu$ to be any probability measure 
on $O(q)$ such that $\Gamma_{\mu}$ is proximal and strongly irreducible. 
For instance we can take any probability $\mu$ with the full support $O(q)$;
then it will be proximal and strongly irreducible,  
since, for $d\geq 3$, the group $O(q)$ is proximal and strongly irreducible on $\mathbb R^d$.
Denote by $\nu$ the unique $\mu$-stationary probability measure on $\mathbb P^{d-1}$. 
Let  $Y$ be the subset of $\mathbb P^{d-1}$ defined as  the set of straight lines in $\mathbb R^d$
which are spanned by vectors $v\in \mathbb R^d$ with $q(v)=0$.
Then the $\mu$-invariant measure $\nu$ is such that $\nu(Y)=1$.
In particular this implies that the support of the measure $\nu$ is contained in $Y$.

Consider the case when $p=1$.
Let $f$ be the linear functional $f: u \in \mathbb R^d \mapsto u_1\in \mathbb R$ 
so that $| f |=1$
and let $y=\mathbb R f \in (\mathbb P^{d-1})^*$. 
Define the algebraic subset $Y=\{ x\in \mathbb P^{d-1} : \delta (x,y) = 1/\sqrt{2} \}$.
We will show that $Y$ contains the support of the measure $\nu$.
Indeed, if $v$ is such that $q(v)=0$, then 
$v_1^2 = v_2^2 + \ldots +v_d^2$ and hence, with
$x =\mathbb R v \in \supp \nu$,
 \begin{align*} 
\delta(x,y) =\frac{|\langle f,v \rangle |}{|f|   |v|} = \frac{1}{\sqrt{2}}.
\end{align*}
Therefore $\nu(Y)=1$. 
Moreover, if we define 
$Y' =\{ x\in \mathbb P^{d-1} : \delta (x,y) = t \}$ with $t \not = 1/\sqrt{2}$, then $\nu(Y')=0$.
\end{example} 


\begin{corollary}
Assume conditions \ref{Cond strong-irred} and \ref{Cond proxim}. 
For any $t \in (-\infty, 0)$, define the hypersurface 
$Y_0 = \{ x \in \mathbb P^{d-1}: \delta (x,y) =t \}$.
Then $\nu(Y_0)$ is $0$ or $1$.
\end{corollary}
\begin{proof}
Let $f\in (\mathbb R^d)^*$ be such that $|f|=1$ and $y = \mathbb R f\in (\mathbb P^{d-1})^*$.
Define the homogenous set $X_0$ as the collection of all vectors $v\in \mathbb R^d$ satisfying $\langle f,v \rangle^2 = t^2|v|^2$.
Then $Y_0$ is the projective image of $X_0\setminus \{0\}$ on $\mathbb P^{d-1}$.
As the function $\phi: v\in \mathbb R^d \mapsto f(v)^2 - t^2 |v|^2 = (f_1v_1+\ldots + f_dv_d)^2 - t^2 (v_1^2+\ldots + v_d^2)$  is a polynomial on $\mathbb R^d$,
the set $X_0$ is algebraic. Since it is also homogeneous, by definition the set $Y_0$ is algebraic in 
$\mathbb P^{d-1}$, and the conclusion follows from Theorem \ref{Lemma-Fursten-set001}. 
%
%
%
\end{proof}

%
%
%

In order to state our second result we need to introduce the transfer operators 
and related notions.
These operators play an important role in many problems related to 
products of random matrices, see for instance 
 Le Page \cite{LeP82} or Guivarc'h and Le Page \cite{GL16}.
We are going to show that the stationary measures related to these operators 
do not charge the algebraic subsets. 
 
For any $g\in\mathbb G$, $x=\mathbb R v \in \mathbb P^{d-1}$ and $y=\mathbb R f \in (\mathbb P^{d-1})^*$  set
\begin{align} \label{cocycle-001}
\sigma(g,x) 
= \log \frac{|g v|}{| v |}, \quad 
\sigma(g^*,y)
= \log \frac{|g^* f|}{| f |}.
\end{align}
Denote $N(g) = \max\{ \| g \|, \| g^{-1} \| \}$.
Consider the sets 
$$ 
I_{\mu}^+ 
=\left\{ s \geq 0 : \int_{\mathbb G} \| g \|^{s} \mu(dg)< + \infty \right\}, \quad 
I_{\mu}^- = \left\{ s \leq 0 : \int_{\mathbb G} N(g)^{-s} \mu(dg) < + \infty \right\} $$
and note that both $I_{\mu}^+$ and $I_{\mu}^-$ contain at least the element $0$. 
Let $s \in I_\mu^+ \cup I_\mu^-$. 
Define the transfer operators $P_s$ and $P_s^*$ as follows: 
for any $\varphi \in \mathcal C(\mathbb P^{d-1})$ and $x\in \mathbb P^{d-1}$, 
\begin{align} \label{DefPs-001a}
P_s \varphi(x) 
=\int_{\mathbb G} e^{s\sigma(g,x)} \varphi(g x) \mu(dg) 
\end{align}
and 
for any $\varphi \in \mathcal C((\mathbb P^{d-1})^* )$ and $y\in (\mathbb P^{d-1})^*$,  
\begin{align} \label{DefPs-001b}
P_s^* \varphi(y) = \int_{\mathbb G} e^{s\sigma(g^*,y)} \varphi(g^* y) \mu(dg). 
\end{align}

The next assumption will be necessary to state the results for negative $s<0$.
\begin{conditionA}[Two sided exponential moment]  \label{Two sided exponential moment} 
There exists 
$\alpha \in(0,1)$ such that
$$
\int_{\mathbb G} N(g)^{\alpha} \mu(dg) < + \infty.
$$
\end{conditionA}

Let $s\in I_{\mu}^+$.
Under conditions \ref{Cond strong-irred} and \ref{Cond proxim}, 
the transfer operator $P_s$ has a unique probability eigenmeasure $\nu_s$ on $\mathbb P^{d-1}$
corresponding to the eigenvalue $\kappa(s)$: 
\begin{align} \label{eq-positives-eigenmesure001}
P_s \nu_s = \kappa(s)\nu_s.
\end{align}
Similarly, the conjugate transfer operator $P_{s}^{*}$ 
has a unique probability eigenmeasure $\nu^*_s$ on $(\mathbb P^{d-1})^*$
corresponding to the same eigenvalue $\kappa(s)$: 
$$
\nu^*_s P_{s}^{*}  = \kappa(s)\nu^*_s.
$$
For detailed account of the mentioned properties for $s>0$ we refer the reader to \cite{GL16},
where it is proved that the mapping $s\mapsto \kappa(s)$ is analytic on a complex neighborhood of the interval $ I_\mu^+$.
Guivarc'h and Le Page \cite{GL16} have also proved that the measure $\nu_s$ does not charge
any proper projective subspace $Y$ in $\mathbb P^{d-1}$.
 
\begin{theorem} [Guivarc'h and Le Page]  \label{Th-main-hyperplane001}
Assume conditions \ref{Cond strong-irred} and \ref{Cond proxim}. 
Then, for any $s\in I_{\mu}^+$
and any proper projective subspace $Y$ of $\mathbb P^{d-1}$ 
it holds that  
$\nu_s(Y) =0$. 
\end{theorem}

Our second result extends Theorem \ref{Th-main-hyperplane001} to proper algebraic subsets of $\mathbb P^{d-1}$.  


\begin{theorem}
\label{Th-main-subsets001}
Assume conditions \ref{Cond strong-irred} and \ref{Cond proxim}. 
Then, for any $s\in I_{\mu}^+$ and any proper algebraic subset $Y$ of $\mathbb P^{d-1}$ 
it holds that either $\nu_s(Y) =0$ or $\nu_s(Y) =1$. 
\end{theorem}

We are able  to prove a similar assertion for  small negative $s<0$.
First we show in Section \ref{sec-stationarymes001} the existence and the uniqueness of the egenmeasure $\nu_s$
for small negative values $s<0$.
\begin{proposition}
\label{Th-main-subsets002}
Assume conditions \ref{Cond strong-irred}, \ref{Cond proxim} and \ref{Two sided exponential moment}. 
Then, there exists $s_0>0$ such that for any negative $s\in [-s_0,0)$, 
the transfer operator $P_s$ has a unique probability eigenmeasure $\nu_s$ on $\mathbb P^{d-1}$
corresponding to the eigenvalue $\kappa(s)$. Similarly, the transfer operator $P_s^*$ 
has a unique probability eigenmeasure $\nu_s^*$ on $(\mathbb P^{d-1})^*$
corresponding to the same eigenvalue $\kappa(s)$. 
\end{proposition}

For negative values of $s$  we can prove the following analogue of  Theorem \ref{Th-main-subsets001}.

\begin{theorem}
\label{proposition-exist001}
Assume conditions \ref{Cond strong-irred}, \ref{Cond proxim} and \ref{Two sided exponential moment}. 
Then, there exists $s_0>0$ such that for any negative $s\in [-s_0,0)$
and any proper algebraic subset $Y$ of $\mathbb P^{d-1}$ 
it holds that either  
$\nu_s(Y) =0$ or $\nu_s(Y) =1$. 
\end{theorem}


%


It is easy to see that all the conclusions of the Example \ref{Example1}
apply also to the measure $\nu_s$.


As an application of the stated results we use Theorem \ref{Lemma-Fursten-set001}
to establish a local limit theorems for 
the coefficients of products of random matrices in $GL(d,\mathbb R)$.

\begin{theorem}\label{Thm_LLT_a}
Assume conditions \ref{Cond strong-irred}, \ref{Cond proxim} and \ref{Two sided exponential moment}. 
Let $- \infty < a_1 < a_2 < \infty$ be real numbers.   
Then, as $n \to \infty$, uniformly in $ f \in (\mathbb R^d)^*$ and $v \in \mathbb R^d$ with $| f | =1$ and $|v|=1$, 
\begin{align*}
\mathbb{P} \Big(  \log |\langle f, G_n v \rangle|  - n \lambda \in [a_1, a_2]  \Big) 
= \frac{a_2 - a_1}{ \sigma \sqrt{2 \pi n} } (1+o(1)). 
\end{align*}
\end{theorem}

 To the best of our knowledge, a local limit theorem  for the coefficients of products of random matrices
 has not been established in the literature so far.
Local limit theorem for sums of independent random variables
 have been studied by many authors:
 we refer the reader to Gnedenko \cite{Gne48}, Stone \cite{Sto65},  Borovkov and Borovkov \cite{BB08}, 
  Breuillard \cite{Bre05}. 
 For the norm cocycle of products of random matrices, local limit theorems have been proved 
 by Le Page \cite{LeP82}, Guivarc'h \cite{Gui15},  Benoist and Quint \cite{BQ16b}. 
The local limit theorems established in \cite{BQ16b} play important role for
 studying stationary measures on finite volume homogeneous spaces, see \cite{BQ13} for details.  

Other potential applications of Theorem \ref{Lemma-Fursten-set001} are the Berry-Esseen bound,
the Edgeworth expansion and the Cram\'er type moderate deviation.
A paper in progress is \cite{XGL20c}, where a local limit theorem with moderate deviations
has been established.
In its turn Theorems \ref{Th-main-subsets001} and \ref{proposition-exist001} 
can be used to establish various limit theorems 
like the large deviation principle and local limit theorem with large deviations for the coefficients \cite{XGL20d}.
\section{Properties of the stationary measure} \label{sec-stationarymes001}


\subsection{The existence and the uniqueness of the probability eigenmeasure for negative $s<0$ } \label{sec-exist eigenmes}

In this section we prove the existence and the uniqueness of the probability eigenmeasure of the transfer operator $P_s$ 
for negative $s<0$. Actually we shall establish it for $s$ in a sufficiently small neighborhood of $0$, i.e. for $|s|<s_0$, for some small $s_0>0,$ from which the result for negative $s\in (-s_0,0)$ follows. 
The idea is as follows. From a very general fixed point theorem due to Brouwer-Schauder-Tychonoff we deduce 
the existence of such a measure. To establish the uniqueness we make use of the  
general results on the perturbation theory of linear operators.



We proceed to state a fixed point theorem for measures in an abstract context. 
Let $X$ be a compact topological space and $\mathcal C(X)$
 be the space of complex valued continuous functions 
on $X$ equipped with the uniform norm.
Denote by $\mathcal C(X)'$ the topological dual space of $\mathcal C(X)$ equipped with the weak-$*$ topology.
Recall that the weak-$*$ topology is the weakest topology on $\mathcal C(X)'$ 
for which the mapping $\nu \in \mathcal C(X)'  \mapsto \nu(\varphi) \in \mathbb C$ 
is continuous for any $\varphi\in \mathcal C(X)$
and that 
by Riesz representation theorem, the space
$\mathcal C (X)'$ coincides with the 
 space of complex valued Borel measures on $X$.
 
\begin{lemma} \label{lemma-existence eigenmes001}
Assume $T: \mathcal C(X)\mapsto \mathcal C(X)$ is a bounded linear operator such that
$T(f) > 0$ for any $f > 0$. 
Then, there exist a constant $\alpha>0$ and a Borel probability measure $\nu_0$ on $X$ such that
$T'\nu_0 = \alpha \nu_0$, where $T': \mathcal C(X)' \mapsto \mathcal C(X)'$ is the adjoint operator of $T$.
\end{lemma}
\begin{proof}
Recall the Brouwer-Schauder-Tychonoff theorem (see \cite[Appendix]{Bonsall-book1962}): 
Let $\mathcal P$ be a convex and compact subset inside a  topological vector space $V$ and
$A: \mathcal P \mapsto \mathcal P$ be a continuous mapping. Then there exists $\nu_0\in \mathcal P$ such that $A\nu_0=\nu_0$.

We shall apply Brouwer-Schauder-Tychonoff theorem with $V=\mathcal C(X)'$.
By Riesz representation theorem 
the space $\mathcal C (X)'$ coincides with the 
space of complex valued Borel measures on $X$.
 Let $\mathcal P$ be the subspace of $\mathcal  C (X)'$  formed by probability measures.
Then the set $\mathcal P$ is convex and by the Banach-Alaoglu theorem it is compact in the weak-$*$ topology.
Since $T1>0$ everywhere, the mapping 
$\nu \in \mathcal P\mapsto \nu(T(1)) \in \mathbb R$ does not vanish on $\mathcal P$.
Note that for any $\varphi \in \mathcal C (X)$ with $\varphi \geq 0$,  
it holds $T'\nu (\varphi) = \nu(T\varphi) \geq 0$ since $T\varphi \geq 0$.   
This allows to define the mapping $A:\mathcal P\mapsto \mathcal P$ by setting, for any $\nu\in \mathcal P,$
$$
A\nu =\frac{T' \nu}{\nu(T1)}.
$$ 
As $T$ is a bounded linear operator on $\mathcal{C}(X)$, the adjoint operator $T'$ is continuous for the weak-$*$ topology.
Since $T1$ is a continuous function on $X$, the mapping $\nu\mapsto \nu(T1)$ is also continuous for the weak-$*$ topology.  
Therefore, the mapping $A$ is continuous for the  weak-$*$ topology.
By the Brouwer-Schauder-Tychonoff theorem $A$ has a fixed point $\nu_0$: $A\nu_0=\nu_0$.
The assertion follows with $\alpha = \nu_0(T(1))$.
\end{proof}

Let $\gamma \in (0,1)$ be a fixed sufficiently small constant.
Consider the Banach space $\mathscr B_{\gamma} $ of $\gamma$-H\"older continuous functions
on $\mathbb P^{d-1}$ endowed with the norm
\begin{align*} 
\| \varphi \|_{\mathscr B_\gamma} = 
 \sup_{x\in \bb P^{d-1} } |\varphi(x)| 
  + \sup_{x \neq x',\ x,x'\in \bb P^{d-1} } \frac{ |\varphi(x)-\varphi(x')| }{ \mathbf{d}(x,x')^{\gamma} }, 
\end{align*}
where $\mathbf{d}(x,x')$ is the $\sin$ of the angle between the vector lines $x=\mathbb Rv $ and $x'=\mathbb R v'$ in $\mathbb P^{d-1}$:
$\mathbf{d}(x,x') =\sqrt{1-  \Big( \frac{\langle v,v' \rangle}{|v| |v'|} \Big)^2 }.$
The topological dual of $\mathscr B_\gamma$ endowed with the induced norm 
is denoted by $\mathscr B'_\gamma$.
Denote by $\mathscr B_\gamma^*$ the Banach space of $\gamma$-H\"older continuous functions
on $(\mathbb P^{d-1})^*$ endowed with the norm
\begin{align*} 
\| \varphi \|_{\mathscr B_\gamma^*} = 
 \sup_{y\in (\mathbb P^{d-1})^* } |\varphi(x)| 
  + \sup_{y \neq y',\  y,y'\in (\bb P^{d-1})^* } \frac{ |\varphi(y)-\varphi(y')| }{ \mathbf{d}(y,y')^{\gamma} }, 
\end{align*}
where $\mathbf{d}(y,y')$ is the 
$\sin$ of the angle between the vector lines $y=\mathbb R f$ and $y'=\mathbb R f'$ in
$\left(\mathbb P^{d-1}\right)^*$:
$\mathbf{d}(y,y') =\sqrt{1-  \Big( \frac{\langle f,f' \rangle}{|f| |f'|} \Big)^2 }.$

Recall that $\mathcal C (\mathbb P^{d-1})$ is the space 
of the continuous complex valued functions on $\mathbb P^{d-1}$
and
$\mathcal C (\mathbb P^{d-1})'$ is the space 
of the complex valued Borel measures on $\mathbb P^{d-1}$.

\begin{lemma} \label{lemma-uniqueness of eigenmeasure001}
Assume conditions \ref{Cond strong-irred}, \ref{Cond proxim} and \ref{Two sided exponential moment}. 
Then, there exists a positive constant $s_0$ 
such that, for any $|s|<s_0$ 
the operator $P_s$ has a
unique probability Borel eigenmeasure $\nu_s$
associated with the unique eigenvalue $\kappa(s)$. 
 \end{lemma}
%
%
\begin{proof}
Note that for $s$ real and close to $0$, $P_s: \mathcal C (\mathbb P^{d-1}) \mapsto \mathcal C (\mathbb P^{d-1}) $ 
is a bounded linear operator. 
Moreover $P_s\varphi > 0$ for any $\varphi >0$ on  $\mathbb P^{d-1}$.
Denote by $P_s' : \mathcal C (\mathbb P^{d-1})' \mapsto \mathcal C (\mathbb P^{d-1})' $ 
the adjoint of the operator $P_s$.
Using  Lemma \ref{lemma-existence eigenmes001}
with $X=\mathbb P^{d-1}$, we get 
that there exists a probability Borel eigenmeasure $\nu_s$ such that
$P_s' \nu_s = \alpha(s) \nu_s$.
By the duality, for any $\varphi \in \mathcal C (\mathbb P^{d-1})$, we have
$\nu_s(P_s \varphi) = P_s' \nu_s (\varphi)$,
so that
\begin{align} \label{eq-nucirc001}
\nu_s(P_s \varphi)  = \alpha(s) \nu_s (\varphi).
\end{align}
The measure $\nu_s$ and the eigenvalue $\alpha(s)$ might a priori not be unique. 
We shall prove the uniqueness using the perturbation theory of linear operators. 

From the results of Le Page \cite{LeP82} it follows that
the operator $P_0$ has a spectral gap property on the Banach space 
$\mathscr B_\gamma,$ for some $\gamma >0.$
By perturbation theory \cite[Theorem III.8]{HH01} 
there exist constants 
 $s_0>0$, $C>0$, $\rho_0 \in (0,1)$ and holomorphic mappings 
 $s \mapsto \theta_s \in \mathscr B'_\gamma$, $s \mapsto r_s \in \mathscr B_\gamma$,
 $s \mapsto \kappa_s \in \mathbb C$ on $(-s_0, s_0)$
 such that $\theta_0=\nu$, $r_0=1$, $\nu(r_s)=1$, $\theta_s(r_s)= 1$,
\begin{align} \label{eq-perturb001}
\theta_s (P_s \varphi) = \kappa(s) \theta_s(\varphi), \quad P_s r_s = \kappa(s) r_s.    
\end{align}
Moreover, one can choose $s_0$ small enough such that $|\kappa(s)| > \rho_0$ and  
\begin{align} \label{eq-perturb002}
\left\| P_s^n \varphi -  \kappa(s)^{n} \theta_s(\varphi) r_s \right\|_{\mathscr B_\gamma} \leq 
 C \rho_0^n \left\|\varphi \right\|_{\mathscr B_\gamma}.
\end{align}
In particular, \eqref{eq-perturb002} implies that, 
for any complex $|s| < s_0$ the complex number $\kappa(s)$ 
is the unique eigenvalue of $P_s$ in $\mathscr B_\gamma$ with modulus strictly larger than $\rho_0$
and the associated eigenspace is $\mathbb C r_s$.
Indeed, let $\varphi\in \mathscr B_\gamma$, $\varphi \not=0$ and $\lambda\in \mathbb C$ be such that $P_s(\varphi) = \lambda\varphi$ and $\lambda \not = \kappa (s)$. Then 
\begin{align*} 
\lambda \theta_s(\varphi) = \theta_s(P_s \varphi) =  \kappa(s) \theta_s(\varphi), 
\end{align*}
hence $\theta_s(\varphi)=0$. 
Therefore,  \eqref{eq-perturb002} gives that
\begin{align*} 
| \lambda |^n \|\varphi \| _{\mathscr B_\gamma} 
=  \| P_s^n \varphi \| _{\mathscr B_\gamma} 
\leq  C \rho_0^n \|\varphi \|_{\mathscr B_\gamma}. 
\end{align*}
This implies that  $|\lambda | \leq \rho_0$,
which means that $\kappa(s)$ is the unique eigenvalue with modulus strictly larger than $\rho_0$.

Let us now show that the eigenspace associated to $\kappa(s)$ is spanned by the function $r_s$.
Indeed, if $\varphi$ is in this eigenspace 
then $\varphi \in  \mathscr B_\gamma $ and $P_s \varphi = \kappa(s) \varphi$. 
Again by \eqref{eq-perturb002}, we have 
\begin{align*} 
| \kappa(s) |^n \|\varphi - \theta_s (\varphi ) r_s  \| _{\mathscr B_\gamma} 
=   \|  P_s^n \varphi - \kappa(s)^{n} \theta_s(\varphi) r_s   \| _{\mathscr B_\gamma} 
\leq  C \rho_0^n \|\varphi \|_{\mathscr B_\gamma}. 
\end{align*}
Since $|\kappa(s)| > \rho_0$, we get $\varphi = \theta_s(\varphi) r_s$.

We shall use the uniqueness property of $\kappa(s)$
to show that $\kappa(s)$ is real and that $r_s$ takes real values for real $s \in (-s_0, s_0)$.
Indeed, as $s$ is real, for any  $\varphi\in \mathscr B_\gamma$
we have $P_s \overline \varphi = \overline{P_s \varphi }$,
which gives
$P_s\overline r_s = \overline{  P_s  r_s  } = \overline \kappa_s \overline r_s  $.
Since $\kappa(s)$ is the unique eigenvalue of $P_s$ with modulus strictly larger that $\rho_0$ 
this proves $\overline \kappa_s = \kappa_s$.
Besides, from the equation
$P_s\overline r_s  = \kappa_s \overline r_s  $ it follows that 
$\overline r_s$ belongs to the eigenspace $\mathbb C r_s$ associated to $\kappa(s)$,
so there exists $z\in \mathbb C$ such that $\overline r_s = z r_s$.
 Since $\nu(r_s) = 1 =  \nu(\overline r_s)$, 
  we get $z=1$ and hence $\overline r_s=r_s$ as required.

Since $r_0=1$, we can assume that $s_0$ is very small such that $r_s$ is strictly positive for real $s \in (-s_0, s_0)$.
We now prove that $\alpha(s)=\kappa(s)$, for real $s \in (-s_0, s_0)$. 
We put $\psi=r_s$ in \eqref{eq-nucirc001} and use the second identity in \eqref{eq-perturb001}
to obtain
\begin{align} \label{eq-nucirc002}
 \alpha(s) \nu_s (r_s) = \nu_s(P_s r_s)  = \kappa(s) \nu_s(r_s). 
\end{align}
Since $r_s>0$ we have $\nu_s (r_s) > 0$, 
which implies that $\alpha(s)= \kappa(s)$ for real valued $s.$

Iterating \eqref{eq-nucirc001} and using the fact that $\alpha(s)= \kappa(s)$, 
we have that, for any $\varphi \in \mathscr B_\gamma $,
\begin{align} \label{eq-nucirc004}
 \nu_s \left(P_s^n \varphi \right)  =  \kappa(s)^{n} \nu_s (\varphi).
\end{align}
From \eqref{eq-nucirc004} and \eqref{eq-perturb002}, 
taking the limit as $n \to \infty$ we obtain that, 
\begin{align*} 
 \nu_s(\varphi) 
 = \nu_s( \left( \theta_s(\varphi) r_s \right) = \theta_s(\varphi) \nu_s ( r_s ),
\end{align*}
from which it follows that, for any $\varphi \in \mathscr B_\gamma$,
\begin{align*} 
\frac{ \nu_s(\varphi)}{\nu_s ( r_s )} = \theta_s(\varphi).
\end{align*}
This proves that the linear functional $\theta_s$ is indeed a non-negative Borel measure,
and that any non-negative Borel measure which is an eigenmeasure of $P_s$ is proportional to 
$\theta_s$. 
%
%
%
%
\end{proof}

To show the uniqueness of the eigenfunction of the operator $P_s$ we need more notation. 
For any real $|s|<s_0$, let $r_s > 0$ be the function introduced in the proof of Lemma \ref{lemma-uniqueness of eigenmeasure001}.
Introduce the operator $Q_s$ by setting, for any $\varphi\in \mathcal C(\mathbb P^{d-1})$,  
\begin{align} \label{defQs001}
Q_s \varphi = \frac{P_s(r_s \varphi) }{ \kappa(s) r_s }.
\end{align}
Then $Q_s$ is a Markov operator, namely, 
$Q_s\varphi \geq 0$ for $\varphi \geq 0$, and  $ Q_s(1) =1.$

\begin{lemma}
For any $\varphi \in \mathcal C (\mathbb P^{d-1}) $, one has 
\begin{align} \label{eq-convQn001}
\lim_{n\to\infty}\left\|Q_s^n (\varphi) - \frac{\nu_s(\varphi r_s)} {\nu_s(r_s) }\right\|_{\infty} = 0 
\end{align}
and 
\begin{align} \label{eq-convPn001}
\lim_{n\to\infty}\left\| \frac{1}{\kappa_s^n} P_s^n (\varphi) - \frac{\nu_s(\varphi)} {\nu_s(r_s) } r_s \right\|_{\infty} = 0 
\end{align}
\end{lemma}
\begin{proof}
First note that from \eqref{eq-perturb002} we have that \eqref{eq-convPn001} holds for any  $\varphi \in \mathscr B_\gamma $.
This implies that \eqref{eq-convQn001} also holds for any  $\varphi \in \mathscr B_\gamma $.
As $Q_s$ is a Markov Markov, it has the norm $1$ in $\mathcal C (\mathbb P^{d-1})$, so $Q_s^n$ is uniformly
bounded in the space of bounded operators of $\mathcal C (\mathbb P^{d-1})$.
Since $\mathscr B_\gamma$ is dense in $\mathcal C (\mathbb P^{d-1})$,
the convergence \eqref{eq-convQn001} holds for any $\varphi \in \mathcal C (\mathbb P^{d-1})$.
This in turn implies that \eqref{eq-convPn001} also holds for $\varphi \in \mathcal C (\mathbb P^{d-1})$. 
\end{proof}


\begin{lemma}\label{lemma-uniqueness of eifunc001}
Assume conditions \ref{Cond strong-irred}, \ref{Cond proxim} and \ref{Two sided exponential moment}. 
Then, for any $|s|<s_0$, the function $r_s$ is the
unique (up to a scaling constant) non-negative continuous eigenfunction of the operator $P_s$.
 \end{lemma}
\begin{proof}
If we take $\varphi $ to be non-negative
non-zero continuous eigenfunction of $P_s$ 
associated to some eigenvalue $\lambda$,
then by \eqref{eq-convPn001} we get
\begin{align} \label{eq-convPn003}
\lim_{n\to\infty}\left\| \frac{1}{\kappa_s^n} \lambda^n \varphi - \frac{\nu_s(\varphi)} {\nu_s(r_s) } r_s \right\|_{\infty} = 0. 
\end{align}
It follows that $\lambda=\kappa(s)$ and
$\varphi = \frac{\nu_s(\varphi)} {\nu_s(r_s) } r_s,$
which means that the function $\varphi$
coincides (up to a scaling constant) with the eigenfunction $r_s$.
\end{proof}

Applying the previous theory to the operator $P_s^*$ and to the adjoint 
projective space $(\mathbb P^{d-1})^*$ we obtain the following:
\begin{lemma} \label{lemma-uniqueness of eigenmeasurestar002}
Assume conditions \ref{Cond strong-irred}, \ref{Cond proxim} and \ref{Two sided exponential moment}. 
Then, there exists a positive constant $s_0$ 
such that, for any real $|s|<s_0$ 
the operator $P_s^*$ has a
unique probability Borel eigenmeasure $\nu_s^*$
and unique (up to a scaling constant) positive continuous eigenfunction $r_s^*$
associated with the same unique eigenvalue $\kappa^*(s)$. 
 \end{lemma}

We will show in Lemma  \ref{Lemma-expleinenfun-s-neg}   that $\kappa^*(s)= \kappa(s)$.

\subsection{The H\"{o}lder regularity of the stationary measure}
In this section we establish the H\"{o}lder regularity 
of the stationary measure $\nu_s$ defined in Lemma \ref{lemma-uniqueness of eigenmeasure001}. 



Note that $\nu_0$ coincides with the stationary measure $\nu$ defined by \eqref{mu station meas}. 
The H\"{o}lder regularity of 
the stationary measure $\nu$ has been established in \cite{GR85} (see also \cite{BL85, Gui90, BQ16b}): 
under conditions \ref{Cond strong-irred}, \ref{Cond proxim} and \ref{Two sided exponential moment}, 
there exists a constant $\alpha > 0$ such that
\begin{align} \label{RegularInver}
\sup_{y \in (\bb{P}^{d-1})^* } \int_{\bb{P}^{d-1} } \frac{1}{ \delta(y, x)^{\alpha} } \nu(dx)
< +\infty. 
\end{align}
By the Frostman lemma (see \cite{Mattila2015}), the assertion \eqref{RegularInver} implies that 
the Hausdorff dimension of the stationary measure $\nu$ is at least $\alpha$. 
As mentioned before, \eqref{RegularInver} plays a crucial role 
for establishing limit theorems such as the law of large numbers and the central limit theorem
for the coefficients $\langle f, G_n v \rangle$ (see \cite{GR85, BL85, Gui90, BQ16b}).
In the following we establish the H\"{o}lder regularity of the stationary measure $\nu_s$
when $s$ is in a small neighborhood of $0$.
The proof is based on \eqref{RegularInver} and the spectral gap properties of the transfer operator $P_s$
established in subsection \ref{sec-exist eigenmes}.  

\begin{proposition}\label{Thm-Regu-s}
Assume conditions \ref{Cond strong-irred}, \ref{Cond proxim} and \ref{Two sided exponential moment}. 
Then, there exist constants $s_0, \alpha > 0$ such that
\begin{align}\label{ReguInverS01}
\sup_{ s\in (-s_0, s_0) } \sup_{y \in (\bb{P}^{d-1})^* } 
  \int_{\bb{P}^{d-1} } \frac{1}{ \delta(y, x)^{\alpha} } \nu_s(dx)
 < +\infty. 
\end{align}
In particular, there exist constants $\alpha, C>0$ such that for any $0< t <1$,
we have
\begin{align}\label{ReguInverS02}
\sup_{ s \in (-s_0, s_0) } \sup_{y \in (\bb{P}^{d-1})^* }
\nu_s  \left( \left\{ x \in \bb{P}^{d-1}:  \delta(y, x)  \leq t  \right\}  \right) \leq C  t^{\alpha}.
\end{align}
\end{proposition}

One implication of the Proposition \ref{Thm-Regu-s} is that the eigenmeasure $\nu_s$
does not charge the projective hyperplanes. Precise formulation follows:
\begin{corollary} \label{corrol-egenmes nus do not charge hyperplanes001}
Assume conditions \ref{Cond strong-irred}, \ref{Cond proxim} and \ref{Two sided exponential moment}.
Then there exists a constant $s_0>0$ such that for any $s \in (-s_0, s_0)$
and any projective hyperplane $Y$ of $\mathbb P^{d-1}$ it holds $\nu_s(Y)=0$.  
\end{corollary}

Before proceeding to proving Proposition \ref{Thm-Regu-s}, let us first recall 
a change of measure formula which will be used in the proof of Proposition \ref{Thm-Regu-s}.  
For any $s \in (-s_0, s_0)$, the family of probability kernels  
$q_{n}^{s}(x,g) = \frac{ e^{s \sigma(g, x) } }{\kappa^{n}(s)}\frac{r_{s}(g x)}{r_{s}(x)},$
$n\geq 1$, satisfies the cocycle property. 
Thus, the probability measures  
$q_{n}^{s}(x,g_{n}\dots g_{1})\mu(dg_1)\dots\mu(dg_n)$ form a projective system 
on $\bb G^{\mathbb{N}}$. 
By the Kolmogorov extension theorem,
there exists a unique probability measure  $\mathbb Q_s^x$ on $\bb G^{\mathbb{N}}$.  
We denote by $\mathbb{E}_{\mathbb Q_s^x}$ the corresponding expectation.
For any measurable function $\varphi$ on $(\mathbb{P}^{d-1} \times \mathbb R)^{n}$, 
it holds that 
\begin{align}\label{Change of measure}
 \frac{1}{ \kappa^{n}(s) r_{s}(x) }
\mathbb{E} \Big[  r_{s}(G_n x) & e^{ s \sigma(G_n, x) }  
\varphi \Big( G_1 x, \sigma(G_1, x), \dots, G_n x, \sigma(G_n, x) \Big) \Big]   \nonumber\\
& \quad 
= \mathbb{E}_{\mathbb{Q}_{s}^{x}} 
\Big[ \varphi \Big( G_1 x, \sigma(G_1, x),\dots, G_n x, \sigma(G_n, x)  \Big) \Big].
\end{align}
Under the changed measure  $\mathbb Q_s^x$, the Markov chain $(G_n x)$ has a unique stationary measure $\pi_s$ defined by 
$\pi_s(\varphi)=\frac{\nu_s(\varphi r_s)}{\nu_s(r_s)}$, for any $\varphi \in \mathcal C(\mathbb P^{d-1})$.

We shall use the following result which has been established in \cite[Lemma 14.11]{BQ16b}. 

\begin{lemma}\label{Lem_Regularity_nus}
Assume conditions \ref{Cond strong-irred}, \ref{Cond proxim}, \ref{Two sided exponential moment}. 
Then, for any $\ee > 0$, 
there exist constants $c_0 > 0$ and $n_0 \geq 1$
such that for all $n \geq k \geq n_0$, $x \in \mathbb P^{d-1}$ and $y \in (\mathbb P^{d-1})^*$, 
\begin{align*} 
\mathbb{P} \Big( \delta (y, G_n x) \leq e^{- \ee k}  \Big) \leq e^{- c_0 k}.
\end{align*}
\end{lemma}

\begin{proof}[Proof of Proposition \ref{Thm-Regu-s}]
\textit{Step 1.} We choose a small enough constant $s_0 > 0$ and we show that 
for any $\ee>0$, there exist $c_1 >0$ and $n_0 \geq 1$ such that
for $n \geq n_0$, 
\begin{align} \label{Ch7RegularStep1}
\sup_{ s \in (-s_0, s_0) }  \sup_{ y \in (\bb{P}^{d-1})^* }  \sup_{ x \in \bb{P}^{d-1} }   
\bb Q_s^x \Big( \delta(y, G_n x) \leq e^{-\ee n} \Big)
\leq e^{-c_1 n}.
\end{align}
To prove this, using \eqref{Change of measure}
and the fact that the eigenfunction $x \mapsto r_s(x)$ is strictly positive and bounded on $\bb{P}^{d-1}$,
uniformly with respect to $s \in (-s_0, s_0)$, 
we get
\begin{align*}
\bb Q_s^x \Big( \delta(y, G_n x) \leq e^{-\ee n} \Big)
& =  \frac{1}{\kappa^{n}(s) r_{s}(x)}  \bb{E}  \left[  e^{s \sigma(G_n, x)}  r_{s}(G_n x)
 \bbm{1}_{ \{ \delta(y, G_n x) \leq e^{-\ee n} \}}  \right]  \nonumber\\
& \leq  \frac{c}{{\kappa^{n}(s)}}  \bb{E} \left[  e^{s \sigma(G_n,x)} 
  \bbm{1}_{ \{ \delta(y, G_n x) \leq e^{-\ee n} \}}  \right].
\end{align*}
By H\"{o}lder's inequality, it follows that 
\begin{align} \label{Ch7Step1Holder 01}
\bb Q_s^x \Big( \delta(y, G_n x) \leq e^{-\ee n} \Big)
\leq \frac{c}{{\kappa^{n}(s)}} \left[ \bb{E} e^{2s \sigma(G_n, x) } \right]^{1/2}
\Big[ \bb{P} \big( \delta(y, G_n x) \leq e^{-\ee n} \big) \Big]^{1/2}. 
\end{align}
It is easy to see that $\bb{E} e^{2s \sigma(G_n, x) } \leq \Big\{ \mathbb{E} \big[ N(g_1)^{2|s|} \big] \Big\}^n.$
Since $\kappa(0) =1$ and the function $\kappa$ is continuous in a small neighborhood of $0$, 
we can choose a sufficiently small constant $s_0 >0$ such that 
\begin{align*}
 \sup_{ s \in (-s_0, s_0) }  \sup_{ x \in \mathbb{P}^{d-1} } 
\frac{1}{{\kappa^{n}(s)}} \left[ \bb{E} e^{2s \sigma(G_n, x)} \right]^{1/2}
\leq  e^{c_2 n}, 
\end{align*}
where $c_2 > 0$ is a constant satisfying $c_2 < c_0/4$ with $c_0$ given in Lemma \ref{Lem_Regularity_nus}. 
This, together with \eqref{Ch7Step1Holder 01} and Lemma \ref{Lem_Regularity_nus}, 
concludes the proof of \eqref{Ch7RegularStep1} with $c_1 = c_0 /4$.

\textit{Step 2.} 
From the definition of $\bb{Q}_s^x$, 
one can verify that for any $x \in \bb{P}^{d-1}$ and $n \geq 1$,  $\pi_s = (\bb{Q}_s^x)^{*n} * \pi_s$,
where $*$ denotes the convolution of two measures. 
Combining this with \eqref{Ch7RegularStep1}, we get that, 
uniformly in $s \in (-s_0, s_0)$ and $y \in (\bb{P}^{d-1})^*$, 
\begin{align} \label{Ch7RegularStep2 a}
\pi_s  \left( \left\{ x \in \bb{P}^{d-1}:  \delta(y, x)  \leq  e^{-\ee n}  \right\}  \right) 
= \int_{\bb{P}^{d-1}} (\bb{Q}_s^x)^{*n} 
\big( \delta(y, G_n x) \leq e^{-\ee n} \big) \pi_s(dx) 
\leq e^{-c_1 n}. 
\end{align}
We denote $B_{n}:= \{x \in \bb{P}^{d-1}: e^{-\ee (n+1)} \leq  \delta(y, x) \leq e^{-\ee n} \}$. 
Choosing $\alpha \in (0, c_1/\ee)$, 
we deduce from \eqref{Ch7RegularStep2 a} that, uniformly in $s \in (-s_0, s_0)$ and $y \in (\bb{P}^{d-1})^*$, 
\begin{align*}
&  \int_{ \bb{P}^{d-1} }  \frac{1}{ \delta(y, x)^{\alpha} } \pi_s(dx) \nonumber\\
& =  \int_{\{x \in \bb{P}^{d-1}: \delta(y, x)^{\alpha} > e^{-\ee n_0} \}}
\frac{1}{ \delta(y, x)^{\alpha} } \pi_s(dx)  
 + \sum_{n=n_0}^{\infty} \int_{B_{n}} \frac{1}{ \delta(y, x)^{\alpha} } \pi_s(dx) \nonumber\\
& \leq    e^{ \ee n_0 \alpha } 
+ \sum_{n=n_0}^{\infty} e^{\ee \alpha (n+1) } e^{ -c_1 n}
< +\infty. 
\end{align*}
This proves \eqref{ReguInverS01} by the relation $\pi_s(\varphi) = \frac{\nu_s(\varphi r_s)}{\nu_s(r_s)}$, 
for any $\varphi \in \mathcal C(\mathbb P^{d-1})$. 
The inequality \eqref{ReguInverS02} is a direct consequence of \eqref{ReguInverS01} by the Markov inequality. 
\end{proof}

Let $s_0$ be small enough. For any real $s$ such that $|s|< s_0$,
$y\in (\mathbb P^{d-1})^*$ and bounded measurable function $\varphi$ on $\mathbb P^{d-1}$ denote 
\begin{align*} 
\nu_s^y(\varphi) 
= \int_{\mathbb P^{d-1}} \varphi(x) \delta(x,y)^s   \nu_s(dx). 
\end{align*}

\begin{corollary} \label{Corol-TVcontin001}
There exists $s_0>0$ such that for any $s\in (-s_0,0)$,
the mapping $y\in (\mathbb P^{d-1} )^* \mapsto \nu_s^y \in (\mathcal C(\mathbb P^{d-1} ))'$ is continuous for the 
total variation norm on $(\mathcal C(\mathbb P^{d-1}) )'$.
\end{corollary}

\begin{proof}
Let $s_0$ be small enough and $s\in (-s_0,0)$. 
Let $y, y' \in (\mathbb P^{d-1})^*$. Since both $\nu_s^{y}$ and $\nu_s^{y'}$ are absolutely continuous 
with respect to $\nu_s$, we have
\begin{align*} 
\| \nu_s^{y}  - \nu_s^{y'} \|_{TV} &= 
\int_{\mathbb P^{d-1}} 
\left| \delta(x,y)^{s}  -  \delta(x,y')^{s}   \right| \nu_s(dx) \\
&=  \int_{\mathbb P^{d-1}}  \frac{ | \delta(x,y')^{-s} - \delta(x,y)^{-s} | }{ \delta(x,y)^{-s} \delta(x,y')^{-s} }  \nu_s(dx). 
\end{align*}
By the H\"older inequality we get
\begin{align*} 
\| \nu_s^{y}  - \nu_s^{y'} \|_{TV}^3  
\leq  &\int_{\mathbb P^{d-1}}   | \delta(x,y')^{-s} - \delta(x,y)^{-s} | ^3 \nu_s(dx)  \\ 
&\times \int_{\mathbb P^{d-1}}   \delta(x,y)^{3s}  \nu_s(dx) 
 \int_{\mathbb P^{d-1}}   \delta(x,y')^{3s}  \nu_s(dx).
\end{align*}
As $y' \to y$, the first term converges to $0$ by the dominated convergence theorem,
whereas the other two terms remain bounded by Proposition \ref{Thm-Regu-s}.
\end{proof}

\subsection{The explicit form of the eigenfunction for negative $s<0$}

We apply the results of the two previous sections to give an explicit form 
of the eigenfunctions of the operators $P_s$ and $P_s^*$  for negative $s<0$. 
Let us recall the corresponding results for $s>0$ which have been established in \cite{GL16}: 
under conditions \ref{Cond strong-irred} and \ref{Cond proxim}, for any $s\in I_{\mu}^+$, 
the functions
\begin{align} \label{eq-eigenfun positive001}
r_{s}(x)    = \int_{(\mathbb P^{d-1})^*}  \delta(x,y)^s \nu^*_{s}(dy),  \quad
r_{s}^*(y) = \int_{\mathbb P^{d-1}}        \delta(x,y)^s \nu_{s}(dx)
\end{align} 
are the unique (up to a scaling constant) non-negative eigenfunctions of the operators $P_s$ and $P_s^*$.
The proof of these expressions 
for $s<0$ is quite different from that in the case $s>0$; 
it requires the H\"{o}lder regularity of the eigenmeasures $\nu_s$ and $\nu_s^*$, which has been  
established in Proposition \ref{Thm-Regu-s}.

First we state the cohomological equation (see \cite{BQ16b}) which will also be useful later on: 
for any $g \in \bb G$, $y=\mathbb R f \in (\mathbb P^{d-1})^*$ and $x=\mathbb R v \in \mathbb P^{d-1}$,
\begin{align} \label{cohomological-001}
\log \delta (y,gx) +\sigma (g,x) = \log \delta (x,g^*y) +\sigma (g^*,y). 
\end{align}
For the ease of the reader we include a short proof of \eqref{cohomological-001}.
By elementary transformations, 
\begin{align*} 
\log \frac{| \langle f, g v  \rangle |}{|f| |v|} = \log \frac{| \langle f, g v  \rangle |}{|f | |g v|} + \log \frac{| g v |}{|v|}    
= \log \delta(y, gx) + \sigma(g,x)
\end{align*}
and, in the same way,
\begin{align*} 
\log \frac{|\langle v, g^* f  \rangle |}{|f| |v|} = \log \frac{| \langle v, g^* f  \rangle |}{ |g^* f| |v|}  + \log \frac{| g^* f |}{|f|}    
=  \log\delta(x, g^* y) + \sigma(g^*,y).
\end{align*}
By the definition of the automorphism $g^*$, we have 
$
\langle f, g v   \rangle = \langle v, g^* f   \rangle,$ 
hence the identity \eqref{cohomological-001} follows.


\begin{lemma} \label{Lemma-expleinenfun-s-neg} 
Assume conditions \ref{Cond strong-irred}, \ref{Cond proxim} and \ref{Two sided exponential moment}. 
Then, there exists a constant $s_0 > 0$ such that for any $s \in (-s_0, 0)$, 
the eigenfunctions $r_s$ and $r_s^*$ are defined (up to a scaling constant) as follows: 
for $x\in \mathbb P^{d-1}$ and $y\in (\mathbb P^{d-1})^*$, 
\begin{align} \label{expleigenfun001}
r_{s}(x)    = \int_{(\mathbb P^{d-1})^*}  \delta(x,y)^s \nu^*_{s}(dy),  \quad
r_{s}^*(y) = \int_{\mathbb P^{d-1}}        \delta(x,y)^s \nu_{s}(dx).
\end{align} 
Moreover, $\kappa^*(s)=\kappa(s)$ for any $s \in (-s_0, 0)$. 
\end{lemma}
\begin{proof}
Let $x\in \mathbb P^{d-1}$. 
By Proposition \ref{Thm-Regu-s}, there exists $s_0 > 0$ such that
\begin{align*} 
\phi_{s}(x)    = \int_{(\mathbb P^{d-1})^*}  \delta(x,y)^s \nu^*_{s}(dy) 
\end{align*}
is well defined and positive for any $s\in (-s_0,0)$.
By Corollary \ref{Corol-TVcontin001}, the function $\phi_s$ is continuous on $\mathbb P^{d-1}$
when $s_0$ is small enough. 
We claim that $P_s\phi_s = \kappa^*(s)\phi_s$ 
with $\kappa^*$ from Lemma \ref{lemma-uniqueness of eigenmeasurestar002}.
Indeed, since $\phi_s$ is uniformly bounded on $\mathbb{P}^{d-1}$, 
using the cohomological identity \eqref{cohomological-001} and Fubini's theorem we get
\begin{align} 
P_{s} \phi_{s}(x) &  = \int_{\bb G}  e^{s \sigma(g, x) } 
\left( \int_{ (\bb P^{d-1})^* }  \delta (gx, y)^s  \nu_{s}^*(dy) \right) \mu(dg)  \nonumber\\
& =  \int_{ (\bb P^{d-1})^* }  \int_{\bb G}  e^{s \sigma(g^*, y) + s \log \delta (x, g^*y) }
  \mu(dg) \nu_{s}^*(dy).  \label{Scalequa 01}  
\end{align}
The function $y\mapsto \delta(x,y)^s$ belongs to the space $L^1(\nu_s^*)$.
As the operator $P_s^*$ is positive and $\nu_s^*$ is a probability eigenmeasure of $P_s^*$,
then $P_s^*$ can be extended to be a bounded operator on $L^1(\nu_s^*)$, which we still denote by $P_s^*$.
Therefore, by \eqref{Scalequa 01}, 
\begin{align*} 
P_{s} \phi_{s}(x)  
 &=  \int_{ (\bb P^{d-1})^* } P_s^* (\delta(x,\cdot )^s) (y)  \nu_{s}^*(dy) \\
 &= \kappa^*(s) \int_{(\mathbb P^{d-1})^*}  \delta(x,y)^s \nu^*_{s}(dy)  
 =  \kappa^*(s) \phi_s(x). 
\end{align*}
By Lemma \ref{lemma-uniqueness of eifunc001}
we get that $\kappa^*(s)=\kappa(s)$ and
$r_s = c \phi_s$, for some constant $c>0$.

The proof for $r_s^*$ is similar and therefore will not be detailed here. 
\end{proof}
To summarize,  
the same relations between $P_s$, $P_s^*$, $\nu_s$, $\nu_s^*$, $r_s$, $r_s^*$ and $\kappa(s)$
hold for small negative and positive $s$: under appropriate conditions 
there exists $s_0>0$ such that for $s\in (-s_0,0)\cup I_{\mu}^+$,
\begin{align} \label{eqt-eigenfincofP001}
P_s \nu_s=\kappa(s) \nu_s,  \quad P_s^* \nu_s^*=\kappa(s) \nu_s^*
\end{align}
and
\begin{align} \label{eqt-eigenfincofP002}
P_s r_s=\kappa(s) r_s,\quad P_s^* r_s^*=\kappa(s) r_s^*.
\end{align}

\subsection{Harmonicity of the invariant measure} \label{sec-harmon001}

Assume that $s\in (-s_0,0)$ $\cup I_\mu^+ $, where $s_0$ is small enough.
For any $y\in (\mathbb P^{d-1})^*$ and bounded measurable function $\varphi$ on $\mathbb P^{d-1}$ denote 
\begin{align} \label{def-pi_s^x-001}
\upsilon_s^y(\varphi) 
= \int_{\mathbb P^{d-1}} \varphi(x) \frac{ \delta(x,y)^s }{r_s^*(y)}  \nu_s(dx).
\end{align}
Due to the regularity of the eigenmeasure $\nu_s$ we have $0< \delta (x,y)\leq 1$, 
 $\nu_s$-a.s.\ on $\mathbb P^{d-1}$ 
(for large $s>0$ this follows from Theorem \ref{Th-main-hyperplane001};  for small $s<0$ it can be deduced from 
 Proposition \ref{Thm-Regu-s}).
Since the eigenfunction $r_s^*$ is bounded and strictly positive on $\mathbb P^{d-1}$, 
the measures $\nu_s$ and $\upsilon_s^y$ are equivalent.

\begin{lemma} \label{lemma-conTVnorm001} The following two assertions hold:

1. Assume conditions \ref{Cond strong-irred} and \ref{Cond proxim}.
Then, for any $s\in I_\mu^+$, the mapping
$y \in (\mathbb P^{d-1})^*\mapsto \upsilon_s^y \in \mathcal C (\mathbb P^{d-1})'$
is continuous in the total variation norm $ \|\cdot \|_{\mathrm{TV}}$.

2. Assume conditions \ref{Cond strong-irred}, \ref{Cond proxim} and \ref{Two sided exponential moment}.
Then, there exists a constant $s_0>0$ such that for any $s\in (-s_0,0)$, 
the mapping 
$y \in (\mathbb P^{d-1})^*\mapsto \upsilon_s^y \in \mathcal C (\mathbb P^{d-1})'$
is continuous in the total variation norm $ \|\cdot \|_{\mathrm{TV}}$. 
\end{lemma}

\begin{proof}
For positive $s>0$ the assertion of the lemma is easily proved due to the continuity 
of the mapping $(x,y)\mapsto \delta(x,y)^s$ (see Lemma 3.5 of \cite{GL16}).
For negative $s<0$ the assertion of the lemma follows from Corollary \ref{Corol-TVcontin001}.
\end{proof}

For any $g\in \mathbb G$ and $y \in (\mathbb P^{d-1})^*$, set 
\begin{align} \label{q-dens-001}
q_s^*(g,y) = \frac{ e^{s\sigma(g^*,y)} }{\kappa(s) } \frac{ r_s^*(g^* y )}{r_s^*(y)}. 
\end{align}
By \eqref{eqt-eigenfincofP002}, for any $y \in (\mathbb P^{d-1})^*,$
the function $q_s^*(g,y)$ is a positive density on $\mathbb G$, i.e.
\begin{align} \label{q-dens-002}
\int_{\mathbb G} q_s^*(g,y) \mu(dg) =1
\end{align}
and, for any $g\in \mathbb G$,
\begin{align} \label{q-dens-003}
q_s^*(g,y) >0.
\end{align}
For any $g\in\mathbb G$, $y\in (\mathbb P^{d-1})^*$ 
and bounded measurable function $\varphi$ on $\mathbb P^{d-1}$, define 
\begin{align} \label{Def_g_upsilon}
g \upsilon_s^y (\varphi) = \int_{\mathbb P^{d-1}} \varphi (g x) \upsilon_s^y (d x). 
\end{align}
The following lemma can be viewed as a generalization of the stationary property 
$\pi_s = \pi_s Q_s $. 
For $s>0$ it has been obtained in \cite[ Lemma 3.6]{GL16}. 
\begin{lemma} \label{lemmaSG-001}
Assume \ref{Cond strong-irred}, \ref{Cond proxim} and $s\in I_\mu^+$,
or \ref{Cond strong-irred}, \ref{Cond proxim}, \ref{Two sided exponential moment} 
and $s\in (-s_0,0)$ with small enough $s_0>0$. 
Then, for any $y\in (\mathbb P^{d-1})^*$ and bounded measurable function $\varphi$ on $\mathbb P^{d-1}$, 
\begin{align*} 
\upsilon_s^y (\varphi) 
= \int_{\mathbb G}  g \upsilon_s^{g^* y} (\varphi)  q_s^*(g,y) \mu(dg).
\end{align*}
\end{lemma}
\begin{proof}
For short we denote $\psi_s^y(x)= \varphi(x) e^{s\delta(x,y)}/r_s^*(y).$
By the definition \eqref{def-pi_s^x-001} of the measure $\upsilon_s^y$, we have
\begin{align*} 
\upsilon_s^y(\varphi)
= \int_{\mathbb P^{d-1}} \varphi(x) \frac{e^{s\log \delta(x,y)}}{r_s^*(y)} \nu_s(dx)
= \int_{\mathbb P^{d-1}} \psi_s^y(x) \nu_s(dx). 
\end{align*}
From the identity $P_s \nu_s(\psi_s^y) = \kappa(s)\nu_s(\psi_s^y),$  it follows that  
\begin{align*} 
\upsilon_s^y(\varphi) 
& = \frac{1}{\kappa(s)}\int_{\mathbb P^{d-1}} P_s \psi_s^y (x) \nu_s(dx) \nonumber\\
& = \frac{1}{ \kappa(s) r_s^*(y)} \int_{\mathbb P^{d-1}} \int_{\mathbb G}    \varphi(gx)   
e^{s\log \delta(gx,y)+ s\sigma(g,x)}   \mu(dg) \nu_s(dx).
\end{align*}
Using the cohomological identity \eqref{cohomological-001}, the Fubini theorem and \eqref{Def_g_upsilon}, we get
\begin{align*} 
\upsilon_s^y(\varphi) 
&= \frac{1}{\kappa(s)r_s^*(y)} \int_{\mathbb G} \int_{\mathbb P^{d-1}}   \varphi(gx)   
 e^{s\log \delta(x,g^*y)+s\sigma(g^*,y)}  \nu_s(dx) \mu(dg) \nonumber\\
&=  \int_{\mathbb G} \frac{e^{s\sigma(g^*,y)}}{ \kappa(s)r_s^*(y) }  \int_{\mathbb P^{d-1}}   \varphi(gx)   
e^{s\log \delta(x,g^*y)}  \nu_s(dx) \mu(dg) \nonumber\\
&=  \int_{\mathbb G} \frac{e^{s\sigma(g^*,y)}}{\kappa(s)} \frac{r_s^*(g^*y)}{r_s^*(y)}  
 g \upsilon_s^{g^*y}(\varphi) \mu(dg)  \nonumber\\
& = \int_{\mathbb G}  q_s^*(g,y) g \upsilon_s^{g^* y} (\varphi) \mu(dg), 
\end{align*}
as desired. 
\end{proof}

\section{Auxiliary statements}


\subsection{Stationary measures on finite extensions} \label{sec-finite extensions}
Let $G$ be locally compact subgroup of $\mathbb G$.
Assume that $H < G$ is a closed subgroup of finite index, which means that the quotient $ G / H$ is a finite set.
 Let $\mu$ be a probability measure on $ G$.
Denote by $\Omega$ the set of infinite sequences $(g_1,g_2,\ldots)$ and equip it with the measure
$\mu^{\otimes \mathbb N^*}$.  
For any $\omega \in\Omega$ set 
\begin{align*} 
\tau(\omega) = \min \{ k\geq 1 :   \  g_k\ldots g_1 \in   H  \}.
\end{align*}
The stopping time $\tau$ is $\mu^{\otimes \mathbb N^*}$-a.s.\ finite, see Lemma 5.5 of \cite{BQ16b}.
Define the mapping $\mathfrak f: \omega\in \Omega \mapsto g_{\tau(\omega)}\ldots g_1 \in  H$.
Let $\mu_{H}$ be the image of the measure 
$\mu^{\otimes \mathbb N^*}$ by the mapping $\mathfrak f$.
We call $\mu_{H}$ the probability measure induced by $\mu$ on the subgroup $ H$. 
From Lemma 5.7 of \cite{BQ16b} we have the following assertion.
\begin{lemma} \label{lemma 5-7bookBQ}
Assume that the probability measure $\mathfrak m$ is $\mu$-stationary on $\mathbb P^{d-1}$. 
Then $\mathfrak m$  is also a $\mu_{H}$-stationary probability measure. 
\end{lemma}

\subsection{Determination of the support of the stationary measure}\label{subsec-support}

Since we will use facts from complex algebraic geometry, we now introduce 
some basic notions on the complex algebraic sets.
We recall that a subset $X$ in $\mathbb C^d$ is algebraic if there exist complex polynomial functions
$p_1,\ldots,p_k$ on $\mathbb C^d$ such that $X= \{v\in \mathbb C^d :  p_1(v)=\ldots =p_k(v)=0 \}$.  
We say that $X$ is homogeneous if for every $t\in \mathbb C $ and $v\in X$ it holds that $tv \in X$. 
The topological space $X$ is called Noetherian
if every non-increasing sequence of closed subsets is eventually constant. 

We also need the projective space of $\mathbb C^d$, which is defined in a similar way: 
$\mathbb P^{d-1}_{\mathbb C}$ 
is the set of elements $x=\mathbb C v$, 
where $v\in \mathbb C^d\setminus \{ 0 \}.$ 
A subset $Y$ of $\mathbb P^{d-1}_{\mathbb C}$ 
is algebraic if there exists an algebraic homogeneous subset
$X$ whose projective image is $Y$.
With this notion, a set $X$ in $\mathbb R^d$ is (real) algebraic if and only if there exists a complex 
algebraic set $X'$ such that $X=X'\cap \mathbb R^d$.   
In the same way, a set $Y$ in $\mathbb P^{d-1}$ is (real) algebraic if and only if there exists a complex 
algebraic set $Y'$ such that $Y=Y'\cap \mathbb P^{d-1}$.

From the Noetherian property 
of the ring of polynomial functions 
it follows that the sets defined as zeros of infinitely many polynomials are also algebraic.
This implies that the algebraic sets are precisely the closed sets 
of the Zariski topology. 
The set $GL(d,\mathbb C)$  is a Zariski open subset of $\mathbb C^{d^2}$, so we can  equip it with the 
Zariski topology of $\mathbb C^{d^2}$. 
The closure of a set $Y$ in the Zariski topology is denoted by  $\mathrm{Zc}(Y)$.

The limit set of the semigroup $\Gamma_\mu$ is 
a subset of $\mathbb P^{d-1}_{\mathbb C}$ defined as follows: 
\begin{align*}
\Lambda (\Gamma_\mu)=\overline {\{ x_g^+ : g\in \Gamma_\mu, \ g \  \mbox{is proximal} \}}.
\end{align*}
It is well known that $\Lambda (\Gamma_\mu)$ is the smallest nonempty closed $\Gamma_\mu$-invariant set
in the projective space $\mathbb  P^{d-1}_{\mathbb C}$
(which means that any nonempty closed $\Gamma_\mu$-invariant set
in the projective space $\mathbb  P^{d-1}_{\mathbb C}$ contains $\Lambda (\Gamma_\mu)$). 

\begin{lemma} \label{lemma-supp of mes nu}
Assume \ref{Cond strong-irred}, \ref{Cond proxim} and $s\in I_\mu^+$,
or \ref{Cond strong-irred}, \ref{Cond proxim}, \ref{Two sided exponential moment} 
and $s\in (-s_0,0)$ with small enough $s_0>0$. 
Then the support of $\nu_s$ is $\Lambda (\Gamma_\mu)$.
\end{lemma}
\begin{proof}
If we apply Lemma \ref{lemma-existence eigenmes001}
with $T = P_s$ and $X=\Lambda (\Gamma_\mu)$, then we obtain that
$P_s$ admits an eigenmeasure which is concentrated on $\Lambda (\Gamma_\mu)$.
By the uniqueness of the eigenmeasure $\nu_s$ 
(see  \eqref{eq-positives-eigenmesure001} and  Lemma \ref{lemma-uniqueness of eigenmeasurestar002}) 
we get $\nu_s(\Lambda (\Gamma_\mu))=1$.
This proves that the support of $\nu_s$ is contained in $\Lambda (\Gamma_\mu)$.

Conversely, let us prove that for any nonempty open subset  $U\subset \Lambda (\Gamma_\mu)$
it holds that $\nu_s(U)>0$. To show this we fix some point $x\in \supp(\nu_s)$.
Since the orbit $\Gamma_\mu x$ is dense in $\Lambda (\Gamma_\mu)$, we can find an integer $n\in \mathbb N$
and $g\in \supp(\mu)^n$  such that $gx\in U$.    
As $\supp(\mu)^n$ is the support of the measure $\mu^{*n}$
we get 
\begin{align*} 
\mathbb P (G_n x \in U)>0.
\end{align*}
This gives 
\begin{align*} 
P_s^n \mathbf 1_{U}  (x) = \mathbb E \left( e^{s\sigma (G_n, x) } \mathbf 1_{\{ G_n x\in U\} } \right) >0.
\end{align*}
Since the function $P_s^n(\mathbf 1_{U} )$ is upper semi-continuous, 
the set 
$$
V=\{x' \in \mathbb P^{d-1}: P_s^n \mathbf 1_{U} (x') > 0 \}
$$ 
is open. 
By assumption $V\cap \supp(\nu_s) $ is nonempty, hence $\nu_s(V)>0$ and $\nu_s(P_s^n\mathbf 1_{U}  ) >0$.    
It follows that $\nu_s(U) = \kappa(s)^{-n} \nu_s(P_s^n\mathbf 1_{U}  ) >0$, as required.
\end{proof}

\begin{remark} \label{rem to lemma-limitset001}
Let $\Gamma_{\mu}^*$ be the adjoint semigroup of $\Gamma_{\mu}$. 
Then assumptions \ref{Cond strong-irred} and \ref{Cond proxim}
hold for $\Gamma_{\mu}^*$ if and only if they hold for $\Gamma_{\mu}$. 
In particular, 
there exists a smallest nonempty closed $\Gamma_{\mu}^*$-invariant subset 
in the dual projective space $(\mathbb P^{d-1})^*$, 
which will be called limit set and denoted by 
$\Lambda(\Gamma_{\mu}^*)$.
\end{remark}

Let $H_{\mu}=\textrm{Zc} (\Gamma_{\mu} )$ be the Zariski closure of $\Gamma_{\mu}$
in $GL(d,\mathbb C)$.
Define  $ X_{\mu} = \textrm{Zc} (\Lambda (\Gamma_\mu) ) = \textrm{Zc} ( \supp  \nu  ).$ 
Note that $X_{\mu}$ is an algebraic subset of 
$\mathbb P^{d-1}_{\mathbb C}$.

\begin{lemma} \label{Lemma-subvarH001}
The algebraic set $X_{\mu}$ is  
the unique closed $H_{\mu}$-orbit
in $\mathbb P^{d-1}_{\mathbb C}$.
In particular, every nonempty Zariski closed $H_\mu$-invariant subset contains $X_\mu$. 
\end{lemma}
\begin{proof} 
It is a general fact in algebraic geometry 
(we refer to Borel \cite[Corollary 1.8, p.53]{Bor91})
that there exists $x_0\in \mathbb P^{d-1}$ such that  $Y_0=H_{\mu} x_0$ is
a Zariski closed orbit. 
We shall prove that $X_\mu = Y_0$. 

The orbit $Y_0$ is also analytically closed and, 
moreover, it is $\Gamma_{\mu}$-invariant
(if $g\in \Gamma_{\mu}$ then $g Y_0= g H_{\mu} x_0 = \{  g g' x_0: g'\in H_{\mu}  \} \subset Y $,
as $g g' \in H_{\mu}$).
Therefore, the minimal closed $\Gamma_{\mu}$-invariant set $\Lambda (\Gamma_{\mu})$ 
is contained in $Y_0$:
$\Lambda (\Gamma_{\mu}) \subset Y_0$. 
This implies that $X_\mu = \rm{Zc}(\Lambda (\Gamma_{\mu})) \subset \rm{Zc} (Y_0) = Y_0$, 
since $Y_0$ is Zariski closed. 
So we have showed that $X_\mu \subset Y_0$.  
 
Now we prove the converse inclusion $Y_0 \subset X_\mu$. 
Since $X_{\mu}$ is $\Gamma_\mu$-invariant Zariski closed subset, it is also $H_{\mu}$-invariant. 
This means that 
$H_{\mu} X_{\mu} \subset X_{\mu}$.
 We know that for any $x\in Y_0$ it holds that $H_\mu x = Y_0$.
 Since $X_\mu \subset Y_0$, we have that $H_\mu x =  Y_0$ for any  $x\in X_\mu$. 
 This implies that $Y_0 =H_\mu x \subset X_{\mu}$, 
which concludes the proof.   
\end{proof}
%
%
%
%
%
%
%
%
%
%


To avoid any confusion with the notion of strong irreducibility of a set of matrices introduced before,
let us recall the notion of irreducibility of an algebraic subset. 
We say that an algebraic subset $Y$ is irreducible if $Y$ cannot be 
represented as the union of two proper closed algebraic subsets of $Y$.
 
\begin{lemma} \label{Lemma-subvarH002}
The algebraic set $X_{\mu}$ is irreducible. 
\end{lemma}

\begin{proof} 
Let $H_{\mu}^0$ be the Zariski connected component of $H_{\mu}$ which contains the unit matrix $\bf e$.
Note that $H_{\mu}^0$ is a finite index Zariski closed subgroup of $H_{\mu}$. 

Let $\mu_0$ be the measure on $H_{\mu}^0$ 
induced by the measure $\mu$ through the mapping $\mathfrak f$ as explained in 
Section \ref{sec-finite extensions}.
By Lemma \ref{lemma 5-7bookBQ}, 
the probabilty measure 
$\nu$  defined by \eqref{mu station meas} is  $\mu_0$-stationary.
By the construction of $\mu_0$ we have $\Gamma_{\mu_0}= \Gamma_\mu \cap H_\mu^0$.
Since any finite index subgroup of a strongly irreducible 
group is still strongly irreducible,
 it follows that $\Gamma_{\mu_0}$ is strongly irreducible.
 It is also proximal since any positive power of an element of $\Gamma_\mu$ 
 is also proximal.
By the way, we note that the irreducibility is not necessarily preserved. 
By the result of Furstenberg \cite{Furst BTSPHS-73} the measure $\mu_0$ has a unique stationary 
probability measure $\nu_0$ on $\mathbb P^{d-1}$
 which (by uniqueness) coincides with 
$\nu$.
This implies that $X_{\mu} 
=X_{\mu_{0}} = \rm{Zc} (\supp \nu_0) = \rm{Zc} (\Lambda (\Gamma_{\mu_0}))$.

Applying Lemma \ref{Lemma-subvarH001} with $\mu_0$ instead of $\mu$, 
we conclude that 
$X_{\mu} = X_\mu^0$ 
is an $H_{\mu}^0$-orbit in $\mathbb P^{d-1}$.
Since $H_{\mu}^0$ is connected (and therefore irreducible as an algebraic set), 
we conclude that $X_\mu$ is irreducible as the image of $H_{\mu}^0$ by the Zariski continuous mapping 
$g\in H_{\mu}^0 \mapsto g x_0 \in  X_\mu$.
\end{proof}

\subsection{The maximum principle}
We shall use repeatedly the following simple fact which we call maximum principle. 
\begin{lemma} \label{lemma-Princ-max}
Let $\phi : (X,P) \mapsto \mathbb R$ be a measurable function on the measurable space $X$ equipped 
with the probability measure $P$ which satisfies 
$\phi(x) \leq \beta$, $P$-a.s.\ and $\int \phi (x) P(dx)  = \beta,$ 
where $\beta \in \mathbb R $ is a real number. 
Then $\phi =\beta$, $P$-a.s.  
\end{lemma}


\section{Proof of Theorem \ref{Lemma-Fursten-set001}} \label{section 2}

We shall prove the following statement which implies Theorem \ref{Lemma-Fursten-set001}. 
First we recall that by the definition of $X_{\mu}$ (see Section \ref{subsec-support}) we have 
$ X_{\mu} = \textrm{Zc} ( \supp  \nu  ),$ so that $\nu( X_\mu)=1$.

\begin{proposition} \label{lemma-Y_0} 
Assume conditions \ref{Cond strong-irred} and \ref{Cond proxim}. 
Then, 
for any proper algebraic subset $Y$ of 
$X_\mu$ it holds $\nu( Y)=0$.
\end{proposition}

%


We need to recall some elementary notions from  algebraic geometry 
(see, for instance Borel \cite{Bor91}). 
We say that a topological space $X$ is irreducible   
if and only if $X$ cannot be written  as $X=X_1 \cup X_2$, where $X_1$ and $X_2$ 
are proper closed subsets of $X$.
Recall that the topological space $X$ is Noetherian
if every non-increasing sequence of closed subsets is eventually constant. 
Any Noetherian topological space can be written as 
finite union of closed irreducible subsets. 

The (combinatorial) dimension of a Noetherian topological space $X$
is the maximum length $l$ of sequences $X_0\subsetneq \ldots \subsetneq \ldots X_l$ of 
distinct irreducible closed sets in $X$:
\begin{align*} 
\dim (X) = \sup \{ l:   X_0\subsetneq \ldots \subsetneq \ldots X_l \subsetneq X \} \in \mathbb N \cup \{\infty\}.
\end{align*}
From this definition it is obvious that if $X$ is irreducible and has finite dimension then any proper closed 
subset of $X$ has strictly smaller dimension. 

It is known  that the projective space $\mathbb P^{d-1}_{\mathbb C}$ is irreducible and Noetherian for the Zariski topology 
with $\dim(\mathbb P^{d-1}_{\mathbb C}) = d-1.$
In the following by the dimension of an algebraic subset of $\mathbb P^{d-1}$ we mean the combinatorial dimension.  

Denote by $\mathfrak A (X_\mu)$ the set of irreducible 
algebraic subsets $Y$ of $X_{\mu}$. 
%
%
Set
\begin{align*} 
d_0 = \min \{ 1\leq r \leq d : 
Y\in \mathfrak A(X_\mu),\ \dim(Y)=r,\  
\nu(Y)>0   \}.
\end{align*}
It is easy to see that $d_0 \leq \dim( X_{\mu}) \leq d-1$. 

\begin{lemma} \label{lemma-algsubs0}
Let $Y$ be an algebraic subset of $\mathbb P^{d-1}_{\mathbb C}$
with $\dim(Y)< d_0$. Then $\nu(Y)=0$.  
\end{lemma}
\begin{proof}
Write $Y$ as the union of 
irreducible algebraic subsets $Y_1,\ldots, Y_r$ of  $\mathbb P^{d-1}_{\mathbb C}$: 
$Y=Y_1\cup \ldots \cup Y_r $ (see Proposition 1.5 of Hartshorne \cite{hartshorne}). 
As each $Y_k$ has dimension strictly less that $d_0$,
we have $\nu(Y_k) =0 $, by the definition of $d_0$. So $\nu(Y)=0$.
\end{proof}

We will establish the following assertion,
which is the key point in the proof of Theorem \ref{Lemma-Fursten-set001}.

 \begin{proposition} \label{lemma-dimension001aa}
The dimension of $X_{\mu}$ is $d_0$:  $\dim( X_{\mu}) = d_0$.
\end{proposition}

The proof of this proposition requires some additional assertions.  For any $c>0$ set
$$
\mathcal W_0(c) = \{ Y\in \mathfrak A(X_\mu) : \ \dim(Y)= d_0, \  \nu(Y) \geq c \}.
$$
We start by showing that this set is finite.
\begin{lemma} \label{lemma-ccc100aa}
For any $c>0$ the set $\mathcal W_0(c)$ is finite. Moreover ${\rm card\,} \mathcal W_0(c)  \leq c^{-1}.$ 
\end{lemma} 
\begin{proof}
Let $Y_1,\cdots,Y_r$ be two by two distinct elements of $W_0(c)$.
For any $1\leq i< j\leq r$, the intersection 
$Y_i \cap Y_j $ is an algebraic subset of dimension strictly smaller than
$d_0$. 
By Lemma \ref{lemma-algsubs0} we have $\nu(Y_i \cap Y_j) = 0.$
This implies that $\nu(Y_1 \cup\cdots \cup Y_r) \geq c r$, so $r \leq c^{-1}$.
\end{proof}
Set
\begin{align} \label{eq-beta001}
\beta = \sup \{ \nu(Y) : 
Y\in \mathfrak A(X_\mu),\ \dim(Y)= d_0  \}.
\end{align}
In the following it is important to show that the supreme in \eqref{eq-beta001} is attained. 
\begin{lemma} \label{lemma-beta attained001}
The following maximum is attained,
\begin{align*} 
\beta = \max \{ \nu(Y) : 
Y\in \mathfrak A(X_\mu),\ \dim(Y)= d_0  \}.
\end{align*}
\end{lemma}
\begin{proof}
As $\beta >0$, we have 
\begin{align*} 
\beta = \sup \{ \nu(Y) : Y\in \mathcal W_0 (\beta/2) \}.
\end{align*}
By Lemma \ref{lemma-ccc100aa}, the set $\mathcal W_0(\beta/2)$
is finite. Therefore, the supreme is attained. 
\end{proof}

All the algebraic sets $Y\in \mathfrak A (X_\mu)$ for which the maximum $\beta$ is realized
are collected in the set
\begin{align*} 
\mathcal W = \{ Y\in \mathfrak A(X_\mu) : \ \dim(Y)= d_0, \  \nu(Y) = \beta \},
\end{align*}
which is finite by Lemma \ref{lemma-ccc100aa} (as a subset of $\mathcal W_0(\beta)$). 

\begin{lemma} \label{lemma-ccc104aa}
The set $\mathcal W $ is $\Gamma_\mu$-invariant, which means that 
for any $g\in \Gamma_\mu$ and any $Y\in \mathcal W$ we have that $g Y \in \mathcal W$. 
\end{lemma}

\begin{proof}
Let $Y$ be an element of $\mathcal W$.
Since the measure $\nu$ is $\mu$-stationary, 
we have
$$
\nu(Y) = \int_{\Gamma_{\mu}} \nu (g^{-1} Y) \mu(dg).
 $$
 For any $g\in \Gamma_{\mu}$ we have that $g^{-1} Y \in \mathfrak A(X_\mu) $
 and $g^{-1} Y$ has dimension $d_0$,
hence $\nu(g^{-1} Y) \leq \beta.$ 
Since
 $$
\beta= \nu(Y) = \int_{\Gamma_{\mu}} \nu (g^{-1} Y) \mu(dg) \leq \beta, 
 $$
by the maximum principle (see Lemma \ref{lemma-Princ-max}) 
we get that $\nu (g^{-1} Y) = \beta$ for $\mu$-almost all $g\in \Gamma_{\mu}$.
As $\mathcal W$ is finite, 
we get $g^{-1} \mathcal W = \mathcal W$ for $\mu$-almost all $g\in \Gamma_{\mu}$.
Since $\mathcal W$ is a finite,  the set $\{ g \in \Gamma_{\mu}: g^{-1} \mathcal W = \mathcal W \}$ is closed.
We have shown that it has $\mu$ measure 1, therefore it contains the $\supp \mu$ (by the definition of the latter).
This means that for any $g\in \supp \mu$ we have $g^{-1}\mathcal W \subset \mathcal W$. 
Since the cardinality of $\mathcal W$ is finite, 
we get that $g^{-1}\mathcal W = \mathcal W$ for any $g\in \supp \mu$.
 By the definition of the support of the measure $\mu$, 
 the last statement implies that $g^{-1} \mathcal W = \mathcal W$ for all $g\in \Gamma_{\mu}$.
   Since $g^{-1} \mathcal W = \mathcal W$ is equivalent to $g \mathcal W = \mathcal W$, for  $g\in \Gamma_{\mu}$, 
 the assertion follows.
 \end{proof}


\begin{proof}[Proof of Proposition \ref{lemma-dimension001aa}]
By Lemma \ref{lemma-ccc104aa}, the set $\mathcal W$ is $\Gamma_{\mu}$-invariant 
($\Gamma_{\mu} \mathcal W\subset \mathcal W $),  
and by Lemma \ref{lemma-ccc100aa}, the set $\mathcal W$ is finite. 
Therefore, the set 
$Z=\bigcup_{Y\in \mathcal W }  Y$
is a Zariski closed $\Gamma_{\mu}$-invariant algebraic subset in $X_{\mu}$.
As $\Gamma_{\mu}$ is Zariski dense in $H_{\mu} =\textrm{Zc} (\Gamma_{\mu} )$,
the algebraic set $Z$ is $H_{\mu}$-invariant.
By Lemma \ref{Lemma-subvarH001}, we have that $X_{\mu}\subset Z$.
On the other hand $Z\subset X_{\mu}$, so  $Z= X_{\mu}$,
which reads as $X_{\mu}=\bigcup_{Y\in \mathcal W }  Y$.
According to Lemma \ref{Lemma-subvarH002}, $X_{\mu}$ is irreducible,
hence we get $X_{\mu}= Y$ for some $Y$ in $\mathcal W$.
In particular, $\dim (Y) = d_0 = \dim(X_{\mu})$, which concludes 
the proof of Proposition \ref{lemma-dimension001aa}.
\end{proof}

Now we show that Proposition \ref{lemma-dimension001aa} implies Proposition \ref{lemma-Y_0}. 


\begin{proof}[Proof of Proposition \ref{lemma-Y_0}]
Let $Y$ be an algebraic subset of $X_\mu$ with $\nu(Y)>0$ and let us show that
$Y=X_\mu$.
By Proposition 1.5 of Hartshorne \cite{hartshorne}
we can decompose the set $Y$ into a union of Zariski closed and 
Zariski irreducible
 subsets $Z_1,\ldots, Z_r$ of $X_\mu$:  
$Y=Z_1 \cup\ldots\cup Z_r$. 
Then, there exists $k$ such that $\nu(Z_k) >0$, and hence $\dim(Z_k) \geq d_0=\dim (X_{\mu})$, 
by Proposition \ref{lemma-dimension001aa}. 
It follows that $Z_k= X_{\mu}$, 
hence $Y =X_\mu$ as claimed by our assertion. 
\end{proof}

Now we prove that Proposition \ref{lemma-Y_0} implies Theorem \ref{Lemma-Fursten-set001}.

\begin{proof}[Proof of Theorem \ref{Lemma-Fursten-set001}]
Let $Y$ be an algebraic subset of the projective space $\mathbb P^{d-1}_{\mathbb C}$. 
If $X_\mu \subset Y$, we have 
$\nu(Y)\geq \nu(X_\mu) =1$. 
Otherwise, $Y\cap X_\mu$ is a proper algebraic subset of $X_\mu$.
Since $\nu(X_\mu)=1,$ we have $\nu(Y)=\nu(Y\cap  X_\mu)=0,$ 
by Proposition  \ref{lemma-Y_0}. 
The assertion of the theorem follows now from the fact that
an algebraic subset of $\mathbb P^{d-1}$ is the intersection of an algebraic subset of 
the projective space $\mathbb P^{d-1}_{\mathbb C}$ with  $\mathbb P^{d-1}$.  
\end{proof}



\section{Proof of Theorems \ref{Th-main-subsets001} and  \ref{proposition-exist001}}

We shall first prove the following: 

\begin{proposition} \label{prop-Y_0nu_s001} 
Assume conditions \ref{Cond strong-irred} and \ref{Cond proxim}. 
Then, for any $s\in I_{\mu}^+$ and
for any proper algebraic subset $Y$ of 
$X_\mu$ it holds $\nu_s( Y)=0$.
\end{proposition}

\begin{proposition} \label{prop-Y_0nu_s002} 
Assume conditions \ref{Cond strong-irred}, \ref{Cond proxim} and \ref{Two sided exponential moment}. 
Then, there exists a constant $s_0>0$ such that for any negative $s\in [-s_0,0)$ and 
for any proper algebraic subset $Y$ of 
$X_\mu$ it holds $\nu_s( Y)=0$.
\end{proposition}

We will see at the end of this section that Propositions \ref{prop-Y_0nu_s001} and \ref{prop-Y_0nu_s002} imply 
Theorems \ref{Th-main-subsets001} and  \ref{proposition-exist001}, respectively.

To establish Propositions \ref{prop-Y_0nu_s001} and \ref{prop-Y_0nu_s002} 
we stick to the proof of Guivarc'h and Le Page \cite{GL16}.
The proofs that we give below will work in  both cases $s\in I_\mu^+$ and $s\in (-s_0,0)$, where 
$s_0>0$ is small enough.  

We need a series of auxiliary statements.
We use the notation 
$\mathfrak A (X_\mu)$ from Section \ref{section 2}  and  $\upsilon_s^y$ from Section \ref{sec-harmon001}. 
 For any $y \in (\mathbb P^{d-1})^*$ define
\begin{align*} 
d_0 = \min \{ 1\leq r \leq d : 
Y\in \mathfrak A(X_\mu),\ \dim(Y)=r,\ \upsilon_s^y(Y)>0   \}. 
\end{align*}
Since the measures $\upsilon_s^y$ and $\nu_s$ are equivalent,
there is no dependence on $y$ in the above definition of $d_0$, so that    
\begin{align*} 
d_0 = \min \{ 1\leq r \leq d : 
Y\in \mathfrak A(X_\mu),\ \dim(Y)=r,\ \nu_s(Y)>0   \}.
\end{align*}
It easy to see that $d_0 \leq \dim( X_{\mu}) \leq d-1$.

\begin{lemma} \label{lemma-algsubs1}
Let $Y$ be an algebraic subset of $\mathbb P^{d-1}_{\mathbb C}$
with $\dim(Y)< d_0$. Then $\nu_s(Y)=0$.  
\end{lemma}
\begin{proof}
The proof being similar to that of Lemma \ref{lemma-algsubs0} is left to the reader. 
\end{proof}


In the sequel we are going to prove the following assertion.  
\begin{proposition} \label{lemma-dimension001}
It holds that $\dim( X_{\mu}) = d_0$.
\end{proposition}
%
%
We shall show below that Proposition \ref{lemma-dimension001} 
implies our 
Propositions \ref{prop-Y_0nu_s001} and \ref{prop-Y_0nu_s002}. 

The proof of Proposition \ref{lemma-dimension001} is based on several lemmas.
For any $c>0$ set
$$
\mathcal W_0^y(c) = \{ Y\in \mathfrak A(X_\mu) :\  \dim(Y)= d_0, \  \upsilon_s^y(Y) \geq c \}.
$$
\begin{lemma} \label{lemma-ccc100}
Let $c>0$ be a constant. Then, for any $y\in (\mathbb P^{d-1)^*}$, 
the set $\mathcal W_0^y(c)$ is finite. Moreover ${\rm card\,} \mathcal W_0^y(c)  \leq c^{-1}.$ 
\end{lemma}

\begin{proof}
Let $Y_1,\cdots,Y_r$ be two by two distinct elements of $W_0^y(c)$.
For any $1\leq i< j\leq r$, the intersection $Y_i \cap Y_j $ is an algebraic set of dimension strictly smaller than
$d_0$, hence by Lemma \ref{lemma-algsubs1}, $\upsilon_s^y(Y_i \cap Y_j) = 0.$
This implies that $\upsilon_s^y(Y_1 \cup\cdots \cup Y_r) \geq c r$, so $r \leq c^{-1}$.
\end{proof}

Proceeding as in the proof of Lemma \ref{lemma-beta attained001}, from  Lemma \ref{lemma-ccc100} it follows that,
for $y \in (\mathbb P^{d-1})^*$,
the number 
\begin{align} \label{eq-defhmu001}
h(y) = \max \big\{ \upsilon_s^y(Y) : Y\in \mathfrak A(X_\mu),\ \dim(Y)= d_0  \big\}
\end{align}
is well defined and the $\max$ is attained.
By Lemma \ref{lemma-conTVnorm001}, 
the function $h: y\in(\mathbb P^{d-1})^* \mapsto \mathbb R$ is continuous 
and thus attains its maximum on $(\mathbb P^{d-1})^*$. 
Therefore, we can define
\begin{align} \label{eq-def of beta}
\beta= \max_{y\in (\mathbb P^{d-1})^*} h(y).
\end{align}

\begin{lemma} \label{lemma-ccc101} 
If $h(y)=\beta$ for some $y\in (\mathbb P^{d-1})^*$, then $h(g^*y)=\beta$ for 
all $g\in \Gamma_\mu$.
In particular $h(y)=\beta$ for all $y\in \Lambda (\Gamma_\mu^*)$.
\end{lemma}
\begin{proof}
Assume that $y\in(\mathbb P^{d-1})^*$ and $h(y)=\beta$. 
Using \eqref{eq-defhmu001} and \eqref{eq-def of beta}, we see that
there exists $Y\in \mathfrak A(X_\mu)$ such that $\upsilon_s^y(Y)=\beta$.
Note that $\mathbf 1_{Y}(g x) = \mathbf 1_{g^{-1}Y}(x)$ for $x\in \mathbb P^{d-1}$. 
By Lemma \ref{lemmaSG-001}, it holds that
\begin{align*} 
\beta = \upsilon_s^y(Y) = \int_{\mathbb G}  \upsilon_s^{g^* y}(g^{-1}Y) q_s^*(g,y) \mu(dg).
\end{align*}
Since $g^{-1}Y$ is an irreducible algebraic subset of dimension $d_0$, we have that for any $g^{-1}\in \Gamma_{\mu}$,
\begin{align} \label{lessthanbeta001}
\upsilon_s^{g^* y}(g^{-1}Y) \leq \beta.
\end{align}
Hence
\begin{align*} 
\beta =
\int_{\mathbb G}  \upsilon_s^{g^* y}(g^{-1}Y) q_s^*(g,y) \mu(dg)
\leq 
\beta \int_{\mathbb G} q_s^*(g,y)\mu (dg) = \beta,
\end{align*}
which implies that
\begin{align} \label{intequalbeta001}
\int_{\mathbb G}  \upsilon_s^{g^* y}(g^{-1}Y) q_s^*(g,y) \mu(dg) = \beta.
\end{align}
From \eqref{lessthanbeta001} and \eqref{intequalbeta001}, 
and the fact that $q_s^*(g,y) \mu(dg)$ is a probability measure,
by the maximum principle (Lemma \ref{lemma-Princ-max}), 
it follows that $\upsilon_s^{g^* y}(g^{-1}Y)=\beta$ for $q_s^*(\cdot,y) d\mu$
almost all $g\in \mathbb G$.
Using the fact that the measures $q_s^*(\cdot,y) d\mu$ and $\mu$ are equivalent,  we obtain that 
$\upsilon_s^{g^* y}(g^{-1}Y)=\beta$ for $\mu$-a.s.\  $g\in \mathbb G$. 
Therefore, $h(g^*y)=\beta$ for $\mu$-a.s.\  $g\in \mathbb G$.
Since the function $h$ is continuous on $(\mathbb P^{d-1})^*$, 
we conclude that $h(g^*y)=\beta$ for
all $g\in \supp \mu$, hence, by iteration,  for all $g\in \Gamma_\mu$.

To prove the second assertion, define
\begin{align} \label{setZ-001}
Z = \big\{ y\in (\mathbb P^{d-1} )^* : \ h(y)=\beta  \big\}.
\end{align}
The set $Z$ is nonempty and, since the function $h$ is continuous, it is closed.
From the first assertion of the lemma 
we have $\Gamma_\mu^* Z \subset Z$,
which means that $Z$ is a $\Gamma_\mu^*$-invariant set. 
Since $\Lambda (\Gamma_\mu^*)$ is the smallest nonempty closed $\Gamma_\mu^*$-invariant set
(by Remark \ref{rem to lemma-limitset001}),
we get
$\Lambda (\Gamma_\mu^*) \subset Z$.  
Therefore, by the definition \eqref{setZ-001} of $Z$ we conclude that $h(y)=\beta$ for all  $y \in \Lambda (\Gamma_\mu^*)$.
\end{proof}

For any $y \in (\mathbb P^{d-1})^*$ we collect all the algebraic sets $Y\in \mathfrak A (X_\mu)$ 
for which $h(y)$ defined by \eqref{eq-defhmu001} is realized, in the set 
\begin{align} \label{eq-def of W001}
\mathcal W(y) = \Big\{ Y\in \mathfrak A(X_\mu) :\ \dim(Y)= d_0, \  \upsilon_s^y(Y) = h(y) \Big\}.
\end{align}
In the same way as in the proof of Lemma \ref{lemma-beta attained001}, 
using Lemma \ref{lemma-ccc100} 
one can show that $\mathcal W(y)$ is a finite set: $\mathrm{card} (\mathcal W(y)) < \infty $.
For any $y\in (\mathbb P^{d-1})^*$, set 
\begin{align} \label{eq-def of ny001}
n(y)=\mathrm{card} ( \mathcal W(y) ).
\end{align}
Note that for any $Y\in  \mathcal W(y) $ we have $\upsilon_s^y(Y)=h(y)=\beta$.
So, by Lemma \ref{lemma-ccc100} we get that $n(y)\leq \beta^{-1}$ for any $y\in (\mathbb P^{d-1})^*$.
Set
\begin{align*} 
r = \max \left\{ n(y) :\  y\in (\mathbb P^{d-1})^*,\ h(y)=\beta \right\}.
\end{align*}
\begin{lemma} \label{lemma-ccc102}
Let $y\in (\mathbb P^{d-1})^*$.  If $h(y)=\beta$ and $n(y)=r$, then  $h(g^*y)=\beta$ and 
$n(g^*y)=r$ for all $g\in \Gamma_\mu$.
In particular $h(y)=\beta$ and $n(y)=r$ for all $y\in \Lambda (\Gamma_\mu^*)$.
\end{lemma}
\begin{proof}
Let $y\in (\mathbb P^{d-1})^*$ be such that $h(y)=\beta$ and let $r=n(y)$.
Define 
\begin{align*} 
h_r(y) = \max \Big\{ \upsilon_s^y(Y) : \ 
 Y= Y_1\cup\cdots\cup Y_{r}, \ 
  Y_k \in \mathfrak A(X_\mu), \dim(Y_k)= d_0, 1\leq k \leq r  \Big\}.
\end{align*}
From \eqref{eq-defhmu001} and  \eqref{eq-def of beta}
it follows that $\upsilon_s^y(Y_k) \leq \beta$, thus
$h_r(y) \leq \beta r$. 
Let $y\in (\mathbb P^{d-1})^*$ such that $h_r(y) = \beta r$.
Then there exists $Y\in \mathfrak A(X_\mu)$ 
such that  $\upsilon_s^y(Y)=\beta r$ and  $Y= Y_1\cup\cdots\cup Y_{r},$ 
where $\dim(Y_k)= d_0$ and $Y_k \in \mathfrak A(X_\mu)$.
Moreover, the sets $Y_1,\cdots, Y_{r} $ are two by two distinct and, 
for any $1\leq i< j\leq r$, 
the intersection 
$Y_i \cap Y_j $ is an algebraic set of dimension strictly smaller than
$d_0$. Hence by Lemma \ref{lemma-algsubs1}, we obtain $\upsilon_s^{y}(Y_i \cap Y_j) = 0.$
Therefore, it holds that $\upsilon_s^{y} (Y_1\cup\cdots\cup Y_{r}) = \beta r$.
Using Lemma \ref{lemmaSG-001} and the bound $\upsilon_s^{y} (g^{-1}(Y_1\cup\cdots\cup Y_{r})) \leq \beta r$,  
we get
\begin{align*} 
\beta r &= \upsilon_s^y (Y_1\cup\cdots\cup Y_{r}) \\
&= \int_{\mathbb G} \upsilon_s^{g^*y} \left( g^{-1}(Y_1\cup\cdots\cup Y_{r}) \right) q_s^*(g,y) \mu(dg)  
\leq \beta r.
\end{align*}
This implies that
\begin{align*} 
\int_{\mathbb G} \upsilon_s^{g^*y} \left( g^{-1}(Y_1\cup\cdots\cup Y_{r}) \right) q_s^*(g,y) \mu(dg)  = \beta r.
\end{align*}
By the maximum principle (Lemma \ref{lemma-Princ-max}), it follows that $q_s^*(\cdot,y) d\mu$-a.s.
\begin{align*} 
\upsilon_s^{g^*y} \left( g^{-1}(Y_1\cup\cdots\cup Y_{r}) \right) = \beta r.
\end{align*}
This means that $h_r(g^*y) =\beta r$ for $q_s^*(\cdot,y) d\mu$ almost all $g\in \mathbb G$. 
Since the measures $q_s^*(\cdot,y) d\mu$ and $\mu$ are equivalent,
we deduce that
$h_r(g^*y) =\beta r$ holds for $\mu$ almost all $g\in \mathbb G$.
Thus we have proved that $h_r(y) =\beta r$ implies $h_r(g^*y) =\beta r$
$\mu$-a.s.\ for $\mu$ almost all $g\in \mathbb G$.
Since the function $h$ is continuous on $(\mathbb P^{d-1})^*$,
it follows that $h_r(g^*y) =\beta r$ for all $g\in \supp \mu$ and hence for $g\in \Gamma_{\mu}$.
This means that $\Gamma_{\mu}^* Z_r \subset Z_r$, where   
\begin{align*} 
Z_r = \left\{ y\in (\mathbb P^{d-1})^* :\  h_r(y)=\beta r \right\}.
\end{align*}
By Lemma \ref{lemma-ccc101}, we know that $\Gamma_{\mu}^* Z \subset Z$,
where $Z$ is defined by \eqref{setZ-001}.
Therefore $\Gamma_\mu^* (Z_r \cap Z)  \subset (Z_r \cap Z).$
Noting that
\begin{align*} 
Z_r\cap Z =  \left\{ y\in (\mathbb P^{d-1})^* :\  h(y)=\beta, \ n(y) = r \right\}, 
\end{align*}
 we get the first assertion of the lemma. 
 
Now we prove the second assertion. 
Since the function $h_r$ is continuous, 
the set $Z_r$ is nonempty and closed. 
We have seen in the proof of Lemma \ref{lemma-ccc101} 
that the set $Z$ is also nonempty and closed.  
Recalling that $\Lambda (\Gamma_\mu^*)$ is the smallest closed $\Gamma_\mu^*$-invariant set 
(by Remark \ref{rem to lemma-limitset001}), 
we obtain $\Lambda (\Gamma_\mu^*) \subset Z_r \cap Z,$
which is precisely the second assertion.
%
%
%
%
%
%
%
%
 \end{proof}


\begin{lemma} \label{lemma-ccc103} 
The mapping
$ y\in \Lambda (\Gamma_\mu^*) \mapsto \mathcal W(y)$ is locally constant.
In particular, the number of distinct values of this mapping is finite, i.e.\ 
\begin{align*} 
\mathrm{card} \left\{ \mathcal W(y) :\ y\in \Lambda (\Gamma_\mu^*) \right\} < \infty.
\end{align*}
\end{lemma} 
\begin{proof}
For any $y \in (\mathbb P^{d-1})^*$, set 
\begin{align*} 
\ee(y) = \max \left\{ \upsilon_s^y(Y) : \  
Y\in \mathfrak A(X_\mu),  \ \dim(Y)= d_0,\   \upsilon_s^y(Y) <\beta  \right\},
\end{align*}
where the maximum is attained (this can be established using Lemma \ref{lemma-ccc100} 
in the same way as in the proof of Lemma \ref{lemma-beta attained001}).
Equip the space $\mathcal C (\mathbb P^{d-1})'$ of complex valued Borel measures on $\mathbb P^{d-1}$
with the total variation norm $ \|\cdot \|_{\mathrm{TV}}$.
Since $\ee(y) <\beta $, 
by the continuity of the mapping 
$y \in (\mathbb P^{d-1})^*\mapsto \upsilon_s^y \in \mathcal C (\mathbb P^{d-1})'$
(see Lemma \ref{lemma-conTVnorm001}),
we get 
\begin{align*}
\ee_0=\sup_{y\in (\mathbb P^{d-1})^*) }  \ee(y) < \beta. 
\end{align*}
Note that any algebraic subset $Y$ has the following property: 
\begin{align} \label{eq-either formula}
\upsilon_s^y(Y)=\beta \quad \mbox{or} \quad   
\upsilon_s^y(Y) \leq \varepsilon_0.
\end{align}

On the compact set $\Lambda (\Gamma_\mu^*)$ 
there exists a neighborhood $U_y \subset \Lambda(\Gamma_\mu)$ of $y$
such that for any $y'\in U_y$,
\begin{align} \label{TV-001}
\|\upsilon_s^y - \upsilon_s^{y'}\|_{\mathrm{TV}} < \beta-\ee_0.
\end{align}
Let $Y_1,\ldots,Y_r$ be the elements of $\mathcal W(y)$ that realize $\beta$, 
i.e.\ $\upsilon_s^y(Y_k)=\beta$, $k=1,\ldots,r$. 
In particular, this implies that 
\begin{align} \label{cupcup001}
\upsilon_s^y(Y_1\cup\ldots\cup Y_r)=\beta r.
\end{align}
Then, by \eqref{TV-001} and \eqref{cupcup001},
\begin{align} \label{cupcup002}
\left| \upsilon_s^y(Y_1\cup\ldots\cup Y_r) - \upsilon_s^{y'}(Y_1\cup\ldots\cup Y_r)  \right| 
= \left|  \beta r - \upsilon_s^{y'}(Y_1\cup\ldots\cup Y_r)  \right| 
< \beta -\ee_0. 
\end{align}
By \eqref{eq-either formula}
we have either $\upsilon_s^{y'}(Y_k)=\beta$ 
or  $\upsilon_s^{y'}(Y_k)\leq \ee_0$, for $k=1,\ldots,r$.
If there exists $Y_k$ such that $\upsilon_s^{y'}(Y_k)\leq \ee_0$,
then this will lead to a contradiction. 
Hence we get that $\upsilon_s^{y'}(Y_k)=\beta$ for $k=1,\ldots,r$, 
and so that $\mathcal W(y') \subset \mathcal W(y).$
In turn, since $\mathrm{card}\,\mathcal W(y) = \mathrm{card}\,\mathcal W(y') = r$, 
we obtain that $\mathcal W(y) = \mathcal W(y')$ for any $y'\in U_{y}$, which means that 
the mapping $ y \mapsto \mathcal W(y)$ is locally constant
on $\Lambda (\Gamma_\mu^*)$.

Since  $\Lambda (\Gamma_\mu^*)$ is a closed set of the projective space, it is also compact.
We know that any locally constant mapping on a compact set has a finite range. 
Therefore, the mapping $y \mapsto \mathcal W(y)$ on $\Lambda (\Gamma_\mu^*)$ 
takes only finite many values. 
This proves the second assertion.
\end{proof}

\begin{lemma} \label{lemma-ccc104}
For any $y\in\Lambda (\Gamma_\mu^*)$, we have  
$\mathcal W (g^*y) = g^{-1} \mathcal W (y)$ 
for $\mu$-a.s.\ all $g\in \Gamma_\mu$.
\end{lemma}
\begin{proof} 
Let $y\in\Lambda (\Gamma_\mu^*)$ and
 $Y_1,\ldots,Y_r$ be two by two distinct elements of $\mathcal W(y)$ that realize $\beta$, 
i.e.\ such that $Y_k \in \mathfrak A(X_\mu)$, $\dim(Y_k)= d_0$ and $\upsilon_s^y(Y_k)=\beta$
for any $k=1,\ldots, r$.
Since the sets $Y_1,\cdots, Y_{r} $ are two by two distinct, 
for any $1\leq i< j\leq r$, 
the intersection 
$Y_i \cap Y_j $ is an algebraic set of dimension strictly smaller than
$d_0$, in view of Lemma \ref{lemma-algsubs1} it holds that $\upsilon_s^{y}(Y_i \cap Y_j) = 0.$
This implies $\beta r = \upsilon_s^y (Y_1\cup\cdots\cup Y_{r})$.
Taking into account Lemma \ref{lemmaSG-001} 
and the bound $\upsilon_s^{y} (g^{-1}(Y_1\cup\cdots\cup Y_{r})) \leq \beta r$, we have
\begin{align*} 
\beta r &= \upsilon_s^y (Y_1\cup\cdots\cup Y_{r}) \\
&= \int_{\mathbb G} \upsilon_s^{g^*y} \left( g^{-1}(Y_1\cup\cdots\cup Y_{r}) \right) q_s^*(g,y) \mu(dg) 
\leq \beta r.
\end{align*}
This implies that
\begin{align*} 
\int_{\mathbb G} \upsilon_s^{g^*y} \left( g^{-1}(Y_1\cup\cdots\cup Y_{r}) \right) q_s^*(g,y) \mu(dg)  = \beta r.
\end{align*}
By the maximum principle (Lemma \ref{lemma-Princ-max}), it follows that $\mu$-a.s.\ on $\mathbb G$, and therefore 
following the same reasoning as in the proof of Lemma \ref{lemma-ccc102} for all $g\in \supp \mu$,
\begin{align*} 
\upsilon_s^{g^*y} \left( g^{-1}(Y_1\cup\cdots\cup Y_{r}) \right)  = \beta r.
\end{align*}
Note that $ \upsilon_s^{g^*y} (g^{-1}Y_i ) \leq \beta$, so that for all $g\in \supp \mu$, 
\begin{align*} 
\beta r = \upsilon_s^{g^*y} \left( g^{-1}(Y_1\cup\cdots\cup Y_{r}) \right)    
\leq  \sum_{i=1}^{r} \upsilon_s^{g^*y} (g^{-1}Y_i ) \leq \beta r.
\end{align*}
This implies that $\upsilon_s^{g^*y} (g^{-1}Y_i )= \beta$ for all $g\in \supp \mu$ and $i =1,\cdots,r.$
Thus $g^{-1} \mathcal W (y) \subset \mathcal W (g^*y)$ for all $g\in \supp \mu$. 
Noticing that both sets have the same cardinality $r$, we obtain that 
$g^{-1} \mathcal W (y) = \mathcal W (g^*y)$ for all $g\in \supp \mu$. 
The desired assertion follows. 
\end{proof}


\begin{proof}[Proof of Proposition \ref{lemma-dimension001}]
The proof is similar to that of Proposition \ref{lemma-dimension001aa}.
Set
\begin{align*} 
\mathcal W =   \bigcup_{y\in \Lambda(\Gamma_{\mu}^*)}  \mathcal W (y). 
\end{align*}
By Lemma \ref{lemma-ccc104}, the set $\mathcal W$ is $\Gamma_{\mu}$-invariant:
 $\Gamma_{\mu} \mathcal W = \mathcal W.$
By Lemma \ref{lemma-ccc103}, the set $\mathcal W$ is finite, 
hence 
$Z: = \bigcup_{Y\in \mathcal W }  Y$
is a nonempty 
$\Gamma_{\mu}$-invariant algebraic subset in $X_{\mu}$.
Since $\Gamma_{\mu}$ is Zariski dense in $H_{\mu} = \rm{Zc} (\Gamma_{\mu})$,
the algebraic set $Z$ is $H_{\mu}$-invariant.
By Lemma \ref{Lemma-subvarH001}
$X_{\mu}$ is the minimal $H_{\mu}$ invariant algebraic subset, so $X_{\mu}\subset Z$.
On the other hand, we have $Z=\bigcup_{Y\in \mathcal W }  Y \subset X_{\mu}$,
therefore, we get $Z= X_{\mu}$.
This reads as $X_{\mu}=\bigcup_{Y\in \mathcal W }  Y$.
From Lemma \ref{Lemma-subvarH002}, we know that $X_{\mu}$ is irreducible,
therefore, we obtain that $X_{\mu}= Y$ for some $Y$ in $\mathcal W$.
In particular, it holds that $\dim (Y) = d_0 = \dim(X_{\mu})$, which concludes 
the proof of Proposition \ref{lemma-dimension001}.
\end{proof}

We now are prepared to prove 
 Propositions \ref{prop-Y_0nu_s001} and \ref{prop-Y_0nu_s002}.
The proof is based on Proposition \ref{lemma-dimension001} 
and the arguments already used in Section \ref{section 2}.


\begin{proof}[Proof of 
Propositions \ref{prop-Y_0nu_s001} and \ref{prop-Y_0nu_s002}]
Let $Y$ be an algebraic subset of the set $X_\mu$ ($Y\subset X_\mu$) with $\nu_s (Y)>0$.
As the set $Y$ can be reducible, we decompose it as a finite union of 
irreducible algebraic subsets. Specifically, 
according to Proposition 1.5 of \cite{hartshorne},
there exist irreducible algebraic subsets
$Z_1,\ldots, Z_r$ 
such that  $Y=Z_1 \cup\ldots\cup Z_r$.
From $\nu_s(Y) >0$, it follows that $\nu_s(Z_k) >0$ for some $1 \leq k \leq r$.  
Recalling that $d_0$ is the minimal dimension of irreducible algebraic subsets $U$ satisfying
$\nu_s(U)>0$, 
we get $\dim(Z_k) \geq d_0$.
By Proposition \ref{lemma-dimension001}, we know that $\dim (X_{\mu}) = d_0$.
Hence we have $\dim(Z_k) \geq d_0= \dim (X_{\mu})$. 
This implies that $Z_k= X_{\mu}$, hence $Y=X_\mu$.
The latter obviously proves
the claims of Propositions \ref{prop-Y_0nu_s001} and \ref{prop-Y_0nu_s002}. 
\end{proof}

We end the section by establishing Theorems \ref{Th-main-subsets001} and \ref{proposition-exist001}.

\begin{proof}[Proof of  Theorems \ref{Th-main-subsets001} and \ref{proposition-exist001}]
Let $Y$ be an algebraic subset of the projective space $\mathbb P^{d-1}_{\mathbb C}$. 
If $X_\mu \subset Y$, we have 
$\nu_s(Y)\geq \nu_s(X_\mu) =1$, by Lemma \ref{lemma-supp of mes nu}. 
Otherwise, $Y\cap X_\mu$ is a proper algebraic subset of $X_\mu$.
Since $\nu_s(X_\mu)=1,$ 
it follows that $\nu_s(Y)=\nu_s(Y\cap  X_\mu)=0,$ 
according to Propositions \ref{prop-Y_0nu_s001} and \ref{prop-Y_0nu_s002} for $s>0$ and $s<0$ respectively. 
The assertion of both theorems follows now from the fact that
an algebraic subset of $\mathbb P^{d-1}$ is the intersection of an algebraic subset of 
the projective space $\mathbb P^{d-1}_{\mathbb C}$ with  $\mathbb P^{d-1}$.  
\end{proof}
\section{Proof of local limit theorems for coefficients}

In this section we show how to apply the $0$-$1$ law for the stationary measure $\nu$ to establish a local limit theorem 
for coefficients $\langle f, G_n x \rangle$ of products of random matrices, 
which to the best of our knowledge cannot be found in the literature.

\subsection{Auxiliary results} \label{secAuxres001}

Let us fix a non-negative density function $\rho$ on $\mathbb{R}$ with $\int_{\mathbb{R}} \rho(u) du = 1$, 
whose Fourier transform $\widehat{\rho}$ is supported on $[-1,1]$. 
For any $0< \ee < 1$, define the rescaled density function $\rho_{\ee}$ by
$\rho_{\ee}(u) = \frac{1}{\ee} \rho(\frac{u}{\ee})$, $u \in \mathbb R,$ 
whose Fourier transform has a compact support on $[-\ee^{-1},\ee^{-1}]$. 
Set $\mathbb{B}_{\ee}(u) = \{u' \in\mathbb{R}: |u' - u| \leq \ee \}$. 
For any non-negative integrable function $\psi$, we define
\begin{align}\label{smoo001}
{\psi}^+_{\ee}(u) = \sup_{u'\in\mathbb{B}_{\ee}(u)} \psi(u') 
\quad  \text{and}  \quad 
{\psi}^-_{\ee}(u) = \inf_{u'\in\mathbb{B}_{\ee}(u)} \psi(u'), 
\quad  u \in \mathbb{R}. 
\end{align}

We need the following smoothing inequality from \cite{GLL17}, 
which gives two-sided bounds for the function $\psi$.

\begin{lemma}  \label{estimate u convo}
Suppose that  $\psi$ is a non-negative integrable function and that 
${\psi}^+_{\ee}$ and ${\psi}^-_{\ee}$ are measurable for any $\ee>0$. 
Then, 
there exists a positive constant $C_{\rho}(\ee)$ with $C_{\rho}(\ee) \to 0$ as $\ee \to 0$,
such that 
\begin{align}
{\psi}^-_{\ee}\!\ast\!\rho_{\ee^2}(u) - 
\int_{|w| \geq \ee} {\psi}^-_{\ee}(u - w) \rho_{\ee^2}(w) dw
\leq  \psi(u) \leq (1+ C_{\rho}(\ee))
{\psi}^+_{\ee}\!\ast\!\rho_{\ee^2}(u),  \quad u \in \mathbb{R}.   \nonumber
\end{align}
\end{lemma}


Define the perturbed operator $P_{it}$ as follows: 
for any $t \in \bb R$ and $\varphi \in \mathcal{C}(\bb P^{d-1})$, 
\begin{align}\label{Def_P_it}
P_{it} (\varphi)(x) = \int_{\bb G} e^{it ( \sigma (g, x) - \lambda ) }  \varphi (g x) \mu(dg),  
\quad  x \in \bb P^{d-1}. 
\end{align}
The following proposition which is taken from \cite{XGL20b}  
will be used in the proof of Theorem \ref{Thm_LLT_a}.
Recall that $\nu$ is the invariant measure of the Markov chain $G_n x$ 
on the projective space $\bb P^{d-1}$.
Let $\varphi$ be a $\gamma$-H\"{o}lder continuous function on $\bb P^{d-1}$. 
Assume that $\psi: \mathbb R \mapsto \mathbb C$
is a continuous function with compact support in $\mathbb{R}$, 
and moreover, $\psi$ is differentiable in a small neighborhood of $0$ on the real line.

\begin{proposition} \label{Prop Rn limit1}
Assume conditions \ref{Cond strong-irred}, \ref{Cond proxim} and \ref{Two sided exponential moment}. 
Then, there exist constants $\delta >0$, $c >0$, $C >0$ such that for all $x\in \mathbb{P}^{d-1}$, 
$|l|\leq \frac{1}{\sqrt{n}}$, $\varphi \in \mathcal{B}_{\gamma}$ and $n \geq 1$, 
\begin{align*} 
&  \left| \sigma \sqrt{n}  \,  e^{ \frac{n l^2}{2 \sigma^2} }
\int_{\mathbb R} e^{-it l n} P^{n}_{it}(\varphi)(x) \psi (t) dt
  - \sqrt{2\pi} \nu(\varphi) \psi(0) \right|   \nonumber\\
& \leq  \frac{ C }{ \sqrt{n} } \| \varphi \|_\gamma 
  + \frac{C}{n} \|\varphi\|_{\gamma} \sup_{|t| \leq \delta} \big( |\psi(t)| + |\psi'(t)| \big)
  + Ce^{-cn} \|\varphi\|_{\gamma} \int_{\bb R} |\psi(t)| dt. 
\end{align*}
\end{proposition}



The explicit dependence of the bound on the target function $\varphi$
as well as the rate of convergence established in Proposition \ref{Prop Rn limit1}
will be used in the proof of Theorem \ref{Thm_LLT_a}. 

\subsection{Proof of Theorem \ref{Thm_LLT_a}}

The goal of this subsection is to establish Theorem \ref{Thm_LLT_a}
using Theorems \ref{Lemma-Fursten-set} and \ref{Lemma-Fursten-set001}, 
Lemma \ref{Lem_Regularity_nus} and Proposition \ref{Prop Rn limit1}. 

\begin{proof}[Proof of Theorem \ref{Thm_LLT_a}]
The basic idea is to decompose the logarithm of the coefficient $\log| \langle f, G_n v \rangle |$
as a sum of the norm cocycle $\sigma (G_n,x)$ 
and of $\log \delta(y,G_n x)$, where $x = \mathbb R v \in \mathbb P^{d-1}$ 
and $y= \mathbb R f \in (\mathbb P^{d-1})^*$ with $|v|=1$ and $|f|=1$, see \eqref{eq-represent001}.
Specifically, 
we have the decomposition: 
\begin{align}\label{ScaProLimAn 01}
J : & =    \sigma \sqrt{2 \pi n} \,  
 \mathbb{P} \Big(  \log |\langle f, G_n v \rangle|  - n \lambda \in [a_1, a_2]  \Big)  \nonumber\\
& =    \sigma \sqrt{2 \pi n} \,     
\mathbb{P} \Big(  S_n^v + \log \delta(y, G_n x) \in [a_1, a_2]  \Big), 
\end{align}
where 
\begin{align*}
S_n^v = \log |G_n v| - n\lambda. 
\end{align*}
For any fixed small constant $0< \eta < 1$, we denote 
\begin{align*}
I_k: = (-\eta k, -\eta(k-1)],  \quad  k \in \mathbb{N}^*. 
\end{align*}
In the sequel, let $\floor{a}$ denote the integral part of $a$. 
In order to apply limit theorems established for the norm cocycle $\log |G_n v|$,
we are led to discretize the function $x \mapsto \log \delta(y, x)$ to obtain that: 
with a sufficiently large constant $C_1 >0$,  
\begin{align*} 
J & =  \sigma  \sqrt{2 \pi n}   \,     
  \mathbb{P} 
  \Big(  S_n^v + \log \delta(y, G_n x) \in [a_1, a_2], \  \log \delta(y, G_n x) \leq -\eta \floor{C_1 \log n} \Big) 
   \nonumber\\
& \quad  +  \sigma  \sqrt{2 \pi n}  \,   
 \sum_{k =1}^{\floor{C_1 \log n}} \mathbb{P} 
 \Big(  S_n^v + \log \delta(y, G_n x) \in [a_1, a_2], \  \log \delta(y, G_n x) \in I_k  \Big)     
  \nonumber\\
& =:  J_1 + J_2. 
\end{align*}

\textit{Upper bound of $J_1$.} 
The term $J_1$ is easily handled by using the fact that, with very small probability, 
the Markov chain $(G_n x)_{n \geq 0}$ stays close to the hyperplane $\ker f$. 
Indeed, by Lemma \ref{Lem_Regularity_nus}, we get that there exists a constant $c_{\eta} >0$ such that
\begin{align}\label{Pf_LLTUpp_J1}
J_1 \leq  \sigma  \sqrt{2 \pi n}  \,  \mathbb{P} \Big( \log \delta(y, G_n x) \leq -\eta \floor{C_1 \log n} \Big)
        \leq  \sigma  \sqrt{2 \pi n}  \,  e^{-c_{\eta} \floor{C_1 \log n} } \to 0,
\end{align}
as $n \to \infty$, since the constant $C_1 >0$ is sufficiently large. 

\textit{Upper bound of $J_2$.} 
Note that on the set $\{\log \delta(y, G_n x) \in I_k\}$, we have
$0 < \log \delta(y, G_n x) + \eta k \leq \eta$. It follows that 
\begin{align*}
J_2 \leq  \sigma  \sqrt{2 \pi n}    \,      
  \sum_{k =1}^{ \floor{C_1 \log n} }
\mathbb{E}  \Big( \mathbbm{1}_{\{ S_n^v - \eta k \in [a_1,  a_2 + \eta] \}} 
                             \mathbbm{1}_{\{ \log \delta(y, G_n x)  \in I_k \}} \Big).
\end{align*}
We denote $\psi_1(u) = \mathbbm{1}_{\{ u \in [a_1,  a_2 + \eta] \}}$, $u \in \mathbb{R}$. 
Recall that $\psi^+_{\ee}(u) = \sup_{ u' \in \mathbb{B}_{\ee}(u) } \psi_1(u')$
is defined by \eqref{smoo001}, for $0 < \ee < 1$. 
Using Lemma \ref{estimate u convo}, we get
\begin{align} \label{ScaProLimAn Bn 01}
& J_2 \leq (1+ C_{\rho}(\ee))
  \sigma  \sqrt{2 \pi n}   \,  
\sum_{k =1}^{ \floor{C_1 \log n} }
\mathbb{E}  \left[ (\psi^+_{\ee}\!\ast\!\rho_{\ee^2})( S_n^v - \eta k )  
\mathbbm{1}_{\{ \log \delta(y, G_n x) \in I_k \}} \right].
\end{align} 
For small enough constant $\ee_1 >0$, 
we define the density function $\bar{\rho}_{\ee_1}$ 
by setting $\bar{\rho}_{\ee_1}(u): = \frac{1}{\ee_1}(1 - \frac{|u|}{\ee_1}) $
for $u \in [-\ee_1, \ee_1]$, and $\bar{\rho}_{\ee_1}(u) = 0$ otherwise.
For any $k \in \mathbb{N}^*$, 
with the notation $\chi_k(u):= \mathbbm{1}_{\{u \in I_k \}}$
and $\chi_{k, \ee_1}^+(u) = \sup_{ u' \in \mathbb{B}_{\ee_1}(u) } \chi_k(u')$,
one can check that 
\begin{align} \label{SmoothIne Holder 01}
\chi_k(u) \leq 
(\chi_{k, \ee_1}^+ * \bar{\rho}_{\ee_1})(u)
\leq \chi_{k, 2\ee_1}^+(u), \quad  u \in \mathbb{R}.
\end{align}
For short, we denote $\varphi_k^y(x) = (\chi_{k, \ee_1}^+ * \bar{\rho}_{\ee_1}) (\log \delta(y,x))$, 
$x \in \bb P^{d-1}$,
which is H\"{o}lder continuous on $\mathbb{P}^{d-1}$. 
Using \eqref{SmoothIne Holder 01} leads to 
\begin{align}  \label{ScalarBn a}
J_2 \leq  (1+ C_{\rho}(\ee))  \sigma   \sqrt{2 \pi n}      
\sum_{k =1}^{ \floor{C_1 \log n} }
\mathbb{E}  \left[ \varphi_k^y (G_n x)
({\psi}^+_{\ee}\!\ast\!\rho_{\ee^2})( S_n^v - \eta k )  \right].  
\end{align}
Denote by 
$\widehat{{\psi}}^+_{\ee}$ the Fourier transform of ${\psi}^+_{\ee}$.
By the Fourier inversion formula, 
\begin{align*}
\psi^+_{\ee}\!\ast\!\rho_{\ee^{2}}(u)
= \frac{1}{2\pi} \int_{\mathbb{R}} e^{itu}
\widehat {\psi}^+_{\ee}(t) \widehat\rho_{\ee^{2}}(t)dt,  \quad  u \in \mathbb R. 
\end{align*}
Substituting $u = S_n^v - \eta k$ 
and using Fubini's theorem, we obtain
\begin{align} \label{Scal Bnxl 01}
J_2 \leq  (1+ C_{\rho}(\ee))
  \sigma   \sqrt{\frac{n}{2\pi}}   
  \sum_{k =1}^{ \floor{C_1 \log n} }  \int_{\mathbb{R}} e^{-it \eta k} 
P^{n}_{it} (\varphi_k^y) (x)
\widehat {\psi}^+_{\ee}(t) \widehat\rho_{\ee^{2}}(t) dt, 
\end{align}
where $P_{it}$ is the perturbed operator defined by \eqref{Def_P_it}. 
We shall apply Proposition \ref{Prop Rn limit1} to deal with each integral in \eqref{Scal Bnxl 01} for fixed $k$. 
Note that $e^{ \frac{C k^2}{n} } \to 1$ as $n \to \infty$, uniformly in $1 \leq k \leq \floor{C_1 \log n}$. 
Since the function $\widehat {\psi}^+_{\ee} \widehat\rho_{\ee^{2}}$ 
is compactly supported in $\mathbb{R}$, 
using Proposition \ref{Prop Rn limit1}
with $\varphi = \varphi_k^y$, $\psi=\widehat {\psi}^+_{ \ee} \widehat\rho_{\ee^{2}}$
and  $l = \frac{\eta k}{n}$, 
we obtain that for any fixed $k \geq 1$, 
as $n \to \infty$, 
uniformly in $ f \in (\mathbb R^d)^*$ and $v \in \mathbb R^d$ with $| f | =1$ and $|v|=1$,
\begin{align*}  
\bigg| \sigma  \sqrt{\frac{n}{2\pi}} 
  \int_{\mathbb{R}} e^{-it \eta k} P^{n}_{it} (\varphi_k^y) (x)
\widehat {\psi}^+_{\ee}(t) \widehat\rho_{\ee^{2}}(t) dt   
- \widehat {\psi}^+_{\ee}(0) \widehat\rho_{\ee^{2}}(0) \nu(\varphi_k^y) \bigg|
\leq \frac{C}{\sqrt{n}} \| \varphi_k^y \|_{\gamma}. 
\end{align*}
Note that $\widehat\rho_{\ee^{2}}(0) =1$ and 
\begin{align*}
\widehat \psi^+_{\ee}(0)
= \int_{\mathbb{R}} \sup_{y'\in\mathbb{B}_{\ee}(u)} \psi^+_{\ee}(u') du
= \int_{\mathbb{R}} \mathbbm{1}_{\{ u \in [a_1 - \ee, a_2 + \eta + \ee] \}} du
= a_2 - a_1 + \eta + 2 \ee. 
\end{align*} 
One can calculate that $\gamma$-H\"{o}lder norm 
$\| \varphi_k^y \|_{\gamma}$
is dominated by 
$C \frac{ e^{ \eta \gamma k} }{ ( 1 - e^{-2\varepsilon_1} )^{\gamma} },$
uniformly in $y \in (\mathbb P^{d-1})^*$.  
Taking sufficiently small $\gamma>0$, we obtain that
the series $\frac{C}{\sqrt{n}} \sum_{k = 1}^{ \floor{C_1 \log n} }  
  \frac{ e^{ \eta \gamma k} }{ ( 1 - e^{-2\varepsilon_1} )^{\gamma} }$
converges to $0$ as $n \to \infty$. 
Consequently, we are allowed to interchange the limit as $n \to \infty$
and the sum over $k$ in \eqref{Scal Bnxl 01} 
to obtain that, uniformly in $ f \in (\mathbb R^d)^*$ and $v \in \mathbb R^d$ with $| f | =1$ and $|v|=1$,
\begin{align} \label{LimsuBn a}
\limsup_{n \to \infty}  J_{2}  
\leq    (1+ C_{\rho}(\ee)) (a_2 - a_1 + \eta + 2 \ee)
 \sum_{k =1}^{\infty} \nu (\varphi_k^y).  
\end{align}
Observe that for any $x \in \bb P^{d-1}$, 
\begin{align} \label{LimsuBn b}
\varphi_k^y (x)
\leq \mathbbm{1}_{ \big\{ \log \delta(y, \cdot) \in I_k \big\} }(x) 
+ \mathbbm{1}_{ \big\{ \log \delta(y, \cdot) \in I_{k,\ee_1}  \big\} }(x),
\end{align}
where $I_{k,\ee_1} = \big(-\eta k -2 \ee_1, -\eta k \big]
\cup \big(-\eta (k-1), -\eta (k-1) + 2 \ee_1 \big]$. 
For the first part in \eqref{LimsuBn b}, we have that for any $y \in (\bb P^{d-1})^*$, 
\begin{align} \label{ScalPosiMainpart}
\sum_{k =1}^{\infty} \nu  \Big( x \in \bb P^{d-1}:  \log \delta(y, x) \in I_k   \Big)
=  1.
\end{align}
For the second part in \eqref{LimsuBn b},
we need to apply Theorem \ref{Lemma-Fursten-set} and the zero-one law for the stationary measure $\nu$ established 
in Theorem \ref{Lemma-Fursten-set001}. 
By the Lebesgue dominated convergence theorem, we get
\begin{align*}
& E: =  \lim_{\ee_1 \to 0}  \sum_{k =1}^{\infty} 
\nu  \Big( x \in \bb P^{d-1}:  \log \delta(y, x) \in I_{k, \ee_1}   \Big) \nonumber\\
& =  \sum_{k =1}^{\infty} \nu  \Big( x \in \bb P^{d-1}:  \log \delta(y, x) = -\eta k  \Big)
   + \sum_{k =1}^{\infty} \nu  \Big( x \in \bb P^{d-1}:  \log \delta(y, x) = -\eta (k-1)  \Big) \nonumber\\
& = 2 \sum_{k =1}^{\infty} \nu  \Big( x \in \bb P^{d-1}:  \log \delta(y, x) = -\eta k  \Big),
\end{align*}
where in the last equality we used Theorem \ref{Lemma-Fursten-set}. 
We are going to apply Theorem \ref{Lemma-Fursten-set001} to prove that $E = 0$. 
In fact, for any  $y \in (\bb P^{d-1})^*$ and 
any set $Y_{y,t} = \{ x \in \bb P^{d-1}:  \log \delta(y,x) = t \}$ with $t \in (-\infty, 0)$, 
by Theorem \ref{Lemma-Fursten-set001} it holds that either $\nu(Y_{y,t}) = 0$ or $\nu(Y_{y,t}) = 1$. 
If $\nu(Y_{y,t}) = 0$ for all $y \in (\bb P^{d-1})^*$ and $t \in (-\infty, 0)$, then clearly we get that $E = 0$. 
If $\nu(Y_{y_0, t_0}) = 1$ for some $y_0 \in (\bb P^{d-1})^*$ and $t_0 \in (-\infty, 0)$,
then we can always choose $0< \eta <1$ in such a way that $- \eta k \neq t_0$ for all $k \geq 1$,
so that we also obtain $E = 0$ for all $y \in (\bb P^{d-1})^*$. 
Therefore, combining this with \eqref{LimsuBn a}, \eqref{LimsuBn b} and \eqref{ScalPosiMainpart}, 
letting first $\ee \to 0$ and then $\eta \to 0$, 
and noting that $C_{\rho}(\ee) \to 0$ as $\ee \to 0$, 
we obtain the upper bound: 
uniformly in $ f \in (\mathbb R^d)^*$ and $v \in \mathbb R^d$ with $| f | =1$ and $|v|=1$,
\begin{align} \label{ScaProLimBn Upper 01}
\limsup_{n \to \infty} J_{2}  \leq  a_2 - a_1. 
\end{align}

\textit{Lower bound of $J_{2}$.} 
Since $0 < \log \delta(y, G_n x) + \eta k \leq \eta$ on the set $\{\log \delta(y, G_n x) \in I_k\}$, we have
\begin{align*}
J_2 \geq  \sigma  \sqrt{2 \pi n}      
  \sum_{k =1}^{ \floor{C_1 \log n} }
\mathbb{E}  \Big[ \mathbbm{1}_{\{ S_n^v - \eta k \in [a_1 + \eta,  a_2 - \eta] \}} 
                             \mathbbm{1}_{\{ \log \delta(y, G_n x) \in I_k \}} \Big].
\end{align*}
We denote $\psi_2(u) = \mathbbm{1}_{\{ u \in [a_1,  a_2 + \eta] \}}$, $u \in \mathbb{R}$, 
and recall that $\psi^-_{\ee}(u) = \inf_{u'\in\mathbb{B}_{\ee}(u)} \psi_2(u')$
is defined by \eqref{smoo001}, for $0 < \ee < 1$. 
By Lemma \ref{estimate u convo}, we get
\begin{align} \label{ScaProLimAn Bn 0134}
J_2 & \geq  \sigma  \sqrt{2 \pi n}   \sum_{k =1}^{ \floor{C_1 \log n} }
\mathbb{E}  \left[ (\psi^-_{\ee}\!\ast\!\rho_{\ee^2})( S_n^v - \eta k )  
\mathbbm{1}_{\{ \log \delta(y, G_n x) \in I_k \}} \right]  \nonumber\\
& \quad  - \sigma  \sqrt{2 \pi n}   \sum_{k =1}^{ \floor{C_1 \log n} }
  \int_{|w| \geq \ee} \mathbb{E}  \left[ \psi^-_{\ee} ( S_n^v - \eta k -w )  
  \mathbbm{1}_{\{ \log \delta(y, G_n x) \in I_k \}} \right] \rho_{\ee^2}(w) dw  \nonumber\\
& =: J_3 - J_4. 
\end{align} 
For any $k \in \mathbb{N}$, 
define $\chi_k(u):= \mathbbm{1}_{\{u \in I_k \}}$
and $\chi_{k, \ee_1}^-(u) = \inf_{ u' \in \mathbb{B}_{\ee_1}(u) } \chi_k(u')$. 
It is easy to verify that   
\begin{align} \label{SmoothIne Holder 0134}
\chi_k(u) \geq 
(\chi_{k, \ee_1}^- * \bar{\rho}_{\ee_1})(u)
\geq \chi_{k, 2\ee_1}^-(u), \quad  u \in \mathbb{R},
\end{align}
where $\bar{\rho}_{\ee_1}$ is the density function introduced in \eqref{SmoothIne Holder 01}. 
For short, we denote $\tilde \varphi_k^y(x) = (\chi_{k, \ee_1}^- * \bar{\rho}_{\ee_1}) (\log \delta(y, x))$, 
$x \in \bb P^{d-1}$, which is H\"{o}lder continuous on $\mathbb{P}^{d-1}$. 

\textit{Lower bound of $J_3$.} 
Using \eqref{SmoothIne Holder 0134}, we get
\begin{align*}  
J_3 \geq  \sigma   \sqrt{2 \pi n}  \, \sum_{k =1}^{ \floor{C_1 \log n} }
\mathbb{E}  \left[ \tilde \varphi_k^y (G_n x)
({\psi}^-_{\ee}\!\ast\!\rho_{\ee^2})( S_n^v - \eta k )  \right].  
\end{align*}
In an analogous way as in the proof of \eqref{LimsuBn a}, 
we obtain 
\begin{align} \label{Low_Boun_J3}
\liminf_{n \to \infty}  J_3  \geq  (a_2 - a_1 - 2\eta - 2 \ee)
 \sum_{k =1}^{\infty} \nu (\tilde \varphi_k^y).  
\end{align}
Proceeding in a similar way as in the proof of the upper bound \eqref{ScaProLimBn Upper 01} for $J_2$, 
using Theorem \ref{Lemma-Fursten-set001}, 
we can obtain the lower bound for $J_3$: 
uniformly in $ f \in (\mathbb R^d)^*$ and $v \in \mathbb R^d$ with $| f | =1$ and $|v|=1$, 
\begin{align} \label{ScaProLimBnLow_J3}
\liminf_{\eta \to 0} \liminf_{\ee \to 0} \liminf_{n \to \infty} J_3 \geq  a_2 - a_1. 
\end{align}

\textit{Upper bound of $J_4$.} 
Note that $\psi^-_{\ee} \leq \psi$,
then it follows from Lemma \ref{estimate u convo} that
$\psi^-_{\ee} 
\leq (1+ C_{\rho}(\ee))\widehat{\psi}^+_{\ee} \widehat{\rho}_{\ee^2}$. 
Moreover, using \eqref{SmoothIne Holder 01}, we get 
$\mathbbm{1}_{\{ \log \delta(y, G_n x) \in I_k \}} \leq  (\chi_{k, \ee_1}^+ * \bar{\rho}_{\ee_1})(G_n x)$.
Similarly to \eqref{Scal Bnxl 01}, 
we have that $J_4$ defined in \eqref{ScaProLimAn Bn 0134} is bounded from above by
\begin{align*} 
(1+ C_{\rho}(\ee)) \sigma  \sqrt{\frac{n}{2\pi}}  
 \sum_{k =1}^{ \floor{C_1 \log n} } \int_{|w|\geq \ee}  
   \left[   \int_{\mathbb{R}} e^{-it(\eta k + w) }  P^{n}_{it} (\varphi_k^y)(x)
\widehat {\psi}^+_{\ee}(t) \widehat\rho_{\ee^{2}}(t) dt \right] \rho_{\ee^2}(w) dw.
\end{align*}
Applying Proposition \ref{Prop Rn limit1} 
with $\varphi = \varphi_k^y$ 
and $\psi = \widehat \psi^+_{\ee} \widehat\rho_{\ee^{2}}$,
it follows from the Lebesgue dominated convergence theorem that 
\begin{align*}
\liminf_{n\to\infty}  J_4  
\leq (1+ C_{\rho}(\ee)) \sum_{k=1}^{ \floor{C_1 \log n} }
\nu (\varphi_k^y)  \widehat \psi^+_{\ee}(0) \widehat\rho_{\ee^{2}}(0)
\int_{|w| \geq \ee} \rho_{\ee^2}(w) dw  
\to  0,
\end{align*}
as $\ee \to 0$.
Combining this with \eqref{ScaProLimAn Bn 0134} and \eqref{ScaProLimBnLow_J3}, we get 
the lower bound for $J_2$: 
uniformly in $ f \in (\mathbb R^d)^*$ and $v \in \mathbb R^d$ with $| f | =1$ and $|v|=1$, 
\begin{align} \label{lowelast001}
\liminf_{n\to\infty} J_2 \geq  a_2 - a_1.
\end{align}
Putting together \eqref{Pf_LLTUpp_J1}, \eqref{ScaProLimBn Upper 01} and \eqref{lowelast001},
we conclude the proof of Theorem \ref{Thm_LLT_a}. 
\end{proof}




\begin{thebibliography}{100}


\bibitem{Aou13} Aoun R.: Transience of algebraic varieties in linear groups-applications to generic Zariski density
 \emph{Annales de l'Institut Fourier}, 63(5): 2049-2080, 2013. 



%
%
%

\bibitem{BQ13} Benoist Y., Quint J. F.: Stationary measures and invariant subsets of homogeneous spaces (II).
  \emph{Journal of the American Mathematical Society}, 26(3): 659-734, 2013.

\bibitem{BQ14-RWPS} Benoist Y., Quint J. F.: Random walks on projective spaces. 
\emph{Composito Mathematica}, 150(9): 1579-1606, 2014. 

\bibitem{BQ16a} Benoist Y., Quint J.~F.: Central limit theorem for linear groups. 
\emph{The Annals of Probability}, 44(2): 1308-1340, 2016.

\bibitem{BQ16b} Benoist Y., Quint J. F.: Random walks on reductive groups. 
\emph{Springer International Publishing}, 2016.

\bibitem{Bonsall-book1962}  Bonsall F. F.: Lectures on some fixed point theorems of functional analysis, 
\emph{Bombay: Tata Institute of Fundamental Research}, 1962.

\bibitem{Bor91} Borel A.: Linear algebraic groups. \emph{Springer Science Business Media}, 1991.

\bibitem{BB08} Borovkov A. A., Borovkov K. A.: Asymptotic analysis of random walks.  
\emph{Cambridge University Press}, 2008.

\bibitem{BL85} Bougerol P., Lacroix J.: Products of random matrices with applications to Schr\"{o}dinger operators.  
 \emph{Birkh\"{a}user Boston}, 1985.
 
\bibitem{BFLM11} Bourgain J., Furman A., Lindenstrauss E., Mozes S.: Stationary measures and equidistribution for orbits of nonabelian semigroups on the torus. \emph{Journal of the American Mathematical Society}, 24(1): 231-280, 2011.

\bibitem{Bre05} Breuillard E.: Distributions diophantiennes et th\'{e}or\`{e}me limite local sur $\mathbb{R}^d$. \emph{Probability Theory and Related Fields},  132(1): 13-38, 2005.


%
%
%
%
%
%
%
%
%
%
%
%

\bibitem{Fur63} Furstenberg H.: Noncommuting random products. 
  \emph{Transactions of the American Mathematical Society}, 108(3): 377-428, 1963.

\bibitem{Furst BTSPHS-73} Furstenberg H.: Boundary theory and stochastic processes on homogeneous spaces. 
\emph{Proc. Symp. Pure Math.}, 26: 193-229, 1973. 

\bibitem{FK60} Furstenberg H., Kesten H.: Products of random matrices. 
\emph{The Annals of Mathematical Statistics}, 31(2): 457-469, 1960.

\bibitem{Gne48} Gnedenko B. V.: On a local limit theorem of the theory of probability.
\emph{Uspekhi Matematicheskikh Nauk}, 3(3): 187-194, 1948.

\bibitem{GG96} Goldsheid I. Y., Guivarc'h Y.: Zariski closure and the dimension of the Gaussian law 
of the product of random matrices.       
\emph{Probability Theory and Related Fields}, 105(1): 109-142, 1996.


\bibitem{GLL17} Grama I., Lauvergnat R., Le Page \'{E}.: Conditioned local limit theorems 
for random walks defined on finite Markov chains. 
\emph{Probability Theory and Related Fields}, 176(1-2): 669-735, 2020.


%
%
%
%




\bibitem{Gui90} Guivarc'h Y.: Produits de matrices al\'{e}atoires et applications 
  aux propri\'{e}t\'{e}s g\'{e}om\'{e}triques des sous-groupes du groupe lin\'{e}aire. 
  \emph{Ergodic Theory and Dynamical Systems}, 10(3): 483-512, 1990.

\bibitem{Gui15} Guivarc'h Y.: Spectral gap properties and limit theorems for some random walks and dynamical systems. \emph{Proc. Sympos. Pure Math}, 89: 279-310, 2015.

%
\bibitem{GL16} Guivarc'h Y., Le Page \'{E}.: Spectral gap properties for linear random walks 
and Pareto's asymptotics for affine stochastic recursions. 
\emph{Annales de l'Institut Henri Poincar\'{e}, Probabilit\'{e}s et Statistiques}, 52(2): 503-574, 2016.

\bibitem{GR85} Guivarc'h Y., Raugi A.: Frontiere de Furstenberg, propri\'{e}t\'{e}s de contraction et th\'{e}or\`{e}mes de convergence. \emph{Probability Theory and Related Fields}, 69(2): 187-242, 1985.


%

\bibitem{hartshorne}  Hartshorne, R.: Algebraic geometry. Vol. 52. \emph{Springer Science and Business Media}, 2013.

\bibitem{HH01} Hennion H., Herv\'{e} L.: Limit theorems for Markov chains and stochastic properties of dynamical systems by quasi-compactness. Vol. 1766 of Lecture Notes in Mathematics. \emph{Springer-Verlag, Berlin,} 2001.

%
%
%
%
%
%
%
%

\bibitem{LeP82} Le Page \'{E}.: Th\'{e}or\`{e}mes limites pour les produits de matrices al\'{e}atoires. 
  In Probability measures on groups. \emph{Springer Berlin Heidelberg}, 258-303, 1982.

%

\bibitem{Mattila2015} 
Mattila P.: Geometry of Sets and Measures in Euclidean Spaces: Fractals and Rectifiability 
(Cambridge Studies in Advanced Mathematics) \emph{Cambridge University Press}, 2015. 



\bibitem{Pet75} Petrov V. V.: Sums of independent random variables. \emph{Springer}, 1975. 

%
%
%
%
%

\bibitem{Sto65} Stone C.: A local limit theorem for nonlattice multi-dimensional distribution functions.
\emph{The Annals of Mathematical Statistics}, 36(2): 546-551, 1965.

%


\bibitem{XGL20a} Xiao H., Grama I., Liu Q.: 
Precise large deviation asymptotics for products of random matrices. 
\emph{Stochastic Processes and their Applications}, 130(9): 5213-5242, 2020. 

 
\bibitem{XGL20b} Xiao H., Grama I., Liu Q.: 
Berry-Esseen bound and precise moderate deviations for products of random matrices. \emph{arXiv preprint} 
	arXiv:1907.02438, 2019. 
	
\bibitem{XGL20c} Xiao H., Grama I., Liu Q.: 
Moderate deviation expansions for the coefficients of random walks on the general linear group. 
\emph{In preparation}, 2020.  
	
\bibitem{XGL20d} Xiao H., Grama I., Liu Q.: 
Large deviation expansions for the coefficients of random walks on the general linear group.
 \emph{In preparation}, 2020.  
	


\end{thebibliography}
\end{document}